\documentclass[11pt]{article}
\usepackage[margin=1in]{geometry} 
\usepackage{amssymb,amsfonts,amsmath,bbm,mathrsfs,stmaryrd}
\usepackage{xcolor}
\usepackage{url}

\usepackage[colorlinks,
             linkcolor=black!75!red,
             citecolor=blue,
             pdftitle={Categorical aspects of compact quantum groups},
             pdfauthor={Alexandru Chirvasitu},
             pdfsubject={rings and algebras},
             pdfkeywords={ra.rings and algebras, ct.category theory, rt.representation theory},
             pdfproducer={pdfLaTeX},
             pdfpagemode=None,
             bookmarksopen=true
             bookmarksnumbered=true]{hyperref}

\usepackage{tikz}
\usetikzlibrary{arrows,decorations.pathreplacing,decorations.markings,shapes.geometric,through,fit,shapes.symbols,positioning}

\usepackage[amsmath,thmmarks,hyperref]{ntheorem}
\usepackage{cleveref}

\crefname{section}{Section}{Sections}
\crefformat{section}{#2Section~#1#3} 
\Crefformat{section}{#2Section~#1#3} 

\crefname{subsection}{\S}{\S\S}
\crefformat{subsection}{#2\S#1#3} 
\Crefformat{subsection}{#2\S#1#3} 

\theoremstyle{change}

\newtheorem{lemma}{Lemma}[subsection]
\newtheorem{proposition}[lemma]{Proposition}

\newtheorem{theorem}[lemma]{Theorem}

\theoremstyle{nonumberplain}

\theoremstyle{change}
\theorembodyfont{\upshape}
\theoremsymbol{\ensuremath{\blacklozenge}}

\newtheorem{definition}[lemma]{Definition}
\newtheorem{example}[lemma]{Example}
\newtheorem{remark}[lemma]{Remark}

\crefname{definition}{definition}{definitions}
\crefformat{definition}{#2definition~#1#3} 
\Crefformat{definition}{#2Definition~#1#3} 

\crefname{lemma}{lemma}{lemmas}
\crefformat{lemma}{#2lemma~#1#3} 
\Crefformat{lemma}{#2Lemma~#1#3} 

\crefname{proposition}{proposition}{propositions}
\crefformat{proposition}{#2proposition~#1#3} 
\Crefformat{proposition}{#2Proposition~#1#3} 

\crefname{example}{example}{examples}
\crefformat{example}{#2example~#1#3} 
\Crefformat{example}{#2Example~#1#3} 

\crefname{remark}{remark}{remarks}
\crefformat{remark}{#2remark~#1#3} 
\Crefformat{remark}{#2Remark~#1#3} 

\crefname{corollary}{corollary}{corollaries}
\crefformat{corollary}{#2corollary~#1#3} 
\Crefformat{corollary}{#2Corollary~#1#3} 

\crefname{theorem}{theorem}{theorems}
\crefformat{theorem}{#2theorem~#1#3} 
\Crefformat{theorem}{#2Theorem~#1#3} 

\crefname{equation}{}{}
\crefformat{equation}{(#2#1#3)} 
\Crefformat{equation}{(#2#1#3)}

\theoremstyle{nonumberplain}
\theoremsymbol{\ensuremath{\blacksquare}}

\newtheorem{proof}{Proof}
\newtheorem{proof of cqgfinpres}{Proof of \Cref{th.cqgfinpres}}
\newtheorem{proof of cstarqgsaft}{Proof of \Cref{th.cstarqgsaft}}

\DeclareMathOperator{\id}{id}

\DeclareMathOperator{\Hom}{Hom}

\DeclareMathOperator{\Functors}{Functors}

\DeclareMathOperator{\End}{End}

\DeclareMathOperator*{\revarinjlim}{{\tikz[baseline=(lim.base)] \draw[->] node[inner sep = 0pt] (lim) {$\lim$} (lim.south west) ++(1pt,-3pt) coordinate (a) (lim.south east) ++(-1pt,-3pt) coordinate (b) (a) -- (b) ;}}
\renewcommand\varinjlim\revarinjlim

\DeclareMathOperator*{\revarprojlim}{{\tikz[baseline=(lim.base)] \draw[<-] node[inner sep = 0pt] (lim) {$\lim$} (lim.south west) ++(1pt,-3pt) coordinate (a) (lim.south east) ++(-1pt,-3pt) coordinate (b) (a) -- (b) ;}}
\renewcommand\varprojlim\revarprojlim

\DeclareMathOperator{\cstar}{\mathrm{C}^*}
\DeclareMathOperator{\cstaroh}{\mathrm{C}^*_0}
\DeclareMathOperator{\wstar}{\mathrm{W}^*}

\DeclareMathOperator{\cqg}{\mathrm{CQG}}
\DeclareMathOperator{\cqgab}{\mathrm{CQG}_{ab}}

\DeclareMathOperator{\cstarqg}{\mathrm{C}^*\mathrm{QG}}
\DeclareMathOperator{\cstarqgab}{\mathrm{C}^*\mathrm{QG}_{ab}}
\DeclareMathOperator{\cstarqs}{\mathrm{C}^*\mathrm{QS}}
\DeclareMathOperator{\cstarohqs}{\mathrm{C}^*_0\mathrm{QS}}
\DeclareMathOperator{\wstarqs}{\mathrm{W}^*\mathrm{QS}}

\DeclareMathOperator{\astar}{\mathrm{Alg}^*}
\DeclareMathOperator{\astaroh}{\mathrm{Alg}^*_0}
\DeclareMathOperator{\astarohqs}{\mathrm{A}^*_0\mathrm{QS}}

\DeclareMathOperator{\univ}{\cat{univ}}
\DeclareMathOperator{\red}{\cat{red}}
\DeclareMathOperator{\alg}{\cat{alg}}
\DeclareMathOperator{\forget}{\cat{forget}}
\DeclareMathOperator{\env}{\cat{env}}
\DeclareMathOperator{\cof}{\cat{cofree}}

\newcommand\Fun\Functors

\newcommand{\define}[1]{{\em #1}}

\newcommand{\cat}[1]{\textsc{#1}}

\newcommand{\qedhere}{\mbox{}\hfill\ensuremath{\blacksquare}}

\title{Categorical aspects of compact quantum groups}
\author{Alexandru Chirvasitu\footnote{UC Berkeley, \url{chirvasitua@math.berkeley.edu}}}

\begin{document}
\maketitle

\begin{abstract}
We show that either of the two reasonable choices for the category of compact quantum groups is nice enough to allow for a plethora of universal constructions, all obtained ``by abstract nonsense'' via the adjoint functor theorem. This approach both recovers constructions which have appeared in the literature, such as the quantum Bohr compactification of a locally compact semigroup, and provides new ones, such as the coproduct of a family of compact quantum groups, and the compact quantum group freely generated by a locally compact quantum space. In addition, we characterize epimorphisms and monomorphisms in the category of compact quantum groups.  
\end{abstract}

\noindent {\em Keywords: compact quantum group, CQG algebra, presentable category, SAFT category, adjoint functor theorem}

\tableofcontents


\section*{Introduction}

Compact quantum groups were introduced in essentially their present form in \cite{Woronowicz1987} (albeit under a different name), and the area has been expanding rapidly ever since. The subject can be viewed as part of Connes' general program \cite{Connes1994} to make ``classical'' notions (spaces, topology, differential geometry) non-commutative: One recasts compact groups as $C^*$-algebras via their algebras of continuous functions, and then removes the commutativity assumption from the definition (see \Cref{subse.cstarqg} for details). Starting with the relatively simple resulting definition, all manner of compact-group-related notions and constructions can then be generalized to the non-commutative setting: Peter-Weyl theory (\cite{Woronowicz1987,MR1616348}), Tannaka-Krein duality and reconstruction (\cite{MR943923,MR1424670}), Pontryagin duality (\cite{MR1059324}), actions on operator algebras (\cite{MR1372527,MR1637425,MR1697607}) and other structures, such as (classical or quantum) metric spaces (\cite{MR2174219,2010arXiv1007.0363Q}) or graphs (\cite{MR1937403}), and so on. This list is not (and cannot be) exhaustive.  

The goal of this paper is to analyze compact quantum groups from a category-theoretic perspective, with a view towards universal constructions. 

Bits and pieces appear in the literature: In \cite{MR1316765}, Wang constructs coproducts in the category $\cstarqg$ (see \Cref{subse.cstarqg}), opposite to that of compact quantum groups (it consists of $C^*$-algebras, which are morally algebras of functions on the non-existent quantum groups). More generally, he constructs other types of colimits (e.g. pushouts) of diagrams with one-to-one connecting morphisms. The fact that this coproduct can be constructed simply at the level of $C^*$-algebras (forgetting about comultiplications) parallels the fact that classically, the underlying space of a categorical product $\prod G_i$ of compact groups $G_i$ is, as a set, just the ordinary Cartesian product. The category of compact groups, however, also admits coproducts, and they are slightly more difficult to construct: One endows the ordinary, discrete coproduct $\coprod G_i$ (i.e. coproduct in the category of discrete groups, also known as the \define{free product} of the $G_i$) with the finest topology making the canonical inclusions $G_j\to\coprod G_i$ continuous, and then takes the Bohr compactification of the resulting topological group. It is natural, then, to ask whether or not coproducts of compact quantum groups exist, or equivalently, whether the category $\cstarqg$ opposite to that of compact quantum groups has products. We will see in \Cref{subse.lim} that this is indeed the case, and moreover, the category is complete (i.e. it has all small limits).   

Another example of universal construction that fits well within the framework of this paper is the notion of quantum Bohr compactification \cite{MR2210362}. One of the main results of that paper is, essentially, that the forgetful functor from compact quantum groups to locally compact quantum semigroups has a left adjoint; remembering that we are always passing from (semi)groups to algebras of functions and hence reversing arrows, this amounts to the existence of a certain \define{right} adjoint (\cite[3.1,3.2]{MR2210362}). \Cref{se.appl} recovers this as one among several right-adjoint-type constructions, such as compact quantum groups ``freely generated by a quantum space'' (as opposed to quantum semigroup; see \Cref{subse.spaces}).  

Most compact-quantum-group-related universal constructions in the literature seem to be of a ``left adjoint flavor'': the already-mentioned colimits in $\cstarqg$, the quantum automorphism groups of, say, \cite{MR1637425}, which are basically initial objects in the category of $\cstarqg$ objects endowed with a coaction on a fixed $C^*$-algebra, etc. By contrast, apart from the Bohr compactification mentioned in the previous paragraph, universal constructions of the right adjoint flavor (limits in $\cstarqg$, or right adjoints to functors with domain $\cstarqg$) appear not to have received much attention. This is all the more surprising for at least two reasons. First, they seem to be more likely to exist than the other kind of universal construction; example: a (unital, say) $C^*$-algebra $A$ endowed with a coassociative map $A\to A\otimes A$ into its minimal tensor square (this would be the object dual to a compact quantum semigroup) always has a compact quantum group (meaning its dual object, as in \Cref{def.cstarqg}) mapping into it universally, but does not, in general, have a compact quantum group receiving a universal arrow from it (\Cref{rem.noleftadj}). Secondly, the representation-theoretic interpretation of limits in $\cstarqg$ is often simpler than that of colimits; see \Cref{prop.limTann} (especially part (a)) and surrounding discussion.

The structure of the paper is as follows:

\Cref{se.prel} recalls the machinery that will be used in the sequel and fixes notations and conventions, introducing the two versions of the category opposite that of compact quantum groups: $\cqg$, consisting of so-called CQG algebras (these are like the algebra of representative functions on a compact group; see \Cref{def.cqg}), and $\cstarqg$, whose objects are analogous to algebras of continuous functions on compact groups (\Cref{def.cstarqg}). 

In \Cref{se.cqgfinpres}, \Cref{th.cqgfinpres} shows that the category $\cqg$ is \define{finitely presentable} (\Cref{subse.finpres}). This technical property will later allow us to reduce the existence of right adjoints for functors defined on $\cqg$ to checking that these functors are cocontinuous, i.e. preserve colimits. This is typically an easy task, as the routine nature of most proofs in \Cref{se.appl} shows.

\Cref{se.cstarqgsaft} proves a property slightly weaker than finite presentability for the category $\cstarqg$ (\Cref{th.cstarqgsaft}). The nice features from the previous section are preserved however, and the same types of results (existence of right adjoints to various functors defined on $\cstarqg$) follow. 

In \Cref{se.appl} we list some of the consequences of the previous two sections. These include the automatic existence of limits in $\cqg$ and $\cstarqg$ (\Cref{subse.lim}), compact quantum groups freely generated by quantum spaces and semigroups (\Cref{subse.spaces,subse.bohr} respectively), and universal Kac type compact quantum groups associated to any given compact quantum group (\Cref{subse.kac}).  

Finally, in \Cref{se.mono} we characterize monomorphisms in the categories $\cqg$ and $\cstarqg$. It turns out that in the former they have to be one-to-one, whereas in the latter being mono is slightly weaker than injectivity (\Cref{prop.mono}). The results are analogous to the fact (\cite[Proposition 9]{MR0260829}) that epimorphisms of compact groups are surjective.  

\subsection*{Acknowledgements} This work is part of my PhD dissertation. I would like to thank my advisor Vera Serganova for all the support, and Piotr So\l tan for helpful discussions on the contents of \cite{MR2210362}.


\section{Preliminaries}\label{se.prel}

All algebraic entities in this paper (algebras, coalgebras, bialgebras, etc.) are complex. A $*$-algebra is, as usual, a complex algebra endowed with a conjugate linear, involutive, algebra anti-automorphism `$*$'. Unless we are dealing with non-unital $C^*$-algebras as in \cref{subse.locallycompact} below, in which case the reader will be warned, algebras are assumed to be unital (and coalgebras are always counital).  

Our main references for the necessary basics on coalgebra, bialgebra and Hopf algebra theory are \cite{MR594432,MR1243637,MR0252485}. The notation pertaining to coalgebras is standard: $\Delta$ for antipodes and $\varepsilon$ for the counit, perhaps adorned with the name of the coalgebra if we want to be more precise (example: $\Delta_C$, $\varepsilon_C$). The same applies to antipodes for Hopf algebras, which are usually denoted by $S$. We use Sweedler notation both for comultiplication, as in $\Delta(c)=c_1\otimes c_2$, and for comodule structures: If $\rho:M\to M\otimes C$ is a right $C$-comodule structure, it will be written as $\rho(m)=m_0\otimes m_1$. All comodules are right, and the category of right comodules over a coalgebra $C$ is denoted by $\mathcal{M}^C$. 

For any comodule $V$ over any coalgebra $H$ (the notation suggests that it will become a Hopf algebra soon), there is a largest subcoalgebra $H(V)$ over which $V$ is a comodule. If the comodule structure map is $\rho:V\to V\otimes H$ and $(e_i)_{i\in I}$ is a basis for $V$,then $H(V)$ is simply the span of the elements $u_{ij}$ defined by \[ \rho(e_j) = \sum_i e_i\otimes u_{ij}. \] We refer to $u_{ij}$ as the \define{coefficients} of the basis $(e_i)$, to $(u_{ij})$ as the \define{coefficient matrix} of the basis, and to $H(V)$ as the \define{coefficient coalgebra} of $V$. The coalgebra structure is particularly simple on coefficients: \[\Delta(u_{ij})=\sum_k u_{ik}\otimes u_{kj},\quad \varepsilon(u_{ij})=\delta_{ij}.\] Henceforth, the standing assumption whenever we mention coefficients and coefficient coalgebras is that the comodule in question is finite-dimensional (that is, $I$ is finite). When $V$ is simple, the coefficients $u_{ij}$ with respect to some basis are linearly independent, and the coefficient coalgebra is a \define{matrix coalgebra}, in the sense that its dual is a matrix algebra. 

\begin{remark}\label{rem.corep}
Note that maps $V\to V\otimes C$ are the same as elements of $V\otimes V^*\otimes C=\End(V)\otimes C$. If $u=(u_{ij})$ is the coefficient matrix of a basis $e_i$, $i=\overline{1,n}$ for $V$ and $\End(V)$ is identified with $M_n$ via the same basis $e_i$, then the element of $\End(V)\otimes C\cong M_n(C)$ corresponding to the coaction is exactly the coefficient matrix $u$. We will often blur the distinction between these two points of view, and might refer to $u$ itself as the comodule structure. 
\end{remark}
 
If in the above discussion $H$ is a Hopf algebra, more can be said: The matrix $S(u_{ij})_{i,j}$ is inverse to $(u_{ij})_{i,j}$. Moreover, giving the dual $V^*$ the usual right $H$-comodule structure \[ \langle f_0,v\rangle f_1 = \langle f,v_0\rangle S(v_1),\quad v\in V,\ f\in V^*, \] the coefficient matrix of the basis dual to $(e_i)$ is precisely $(S(u_{ji}))_{i,j}$ (note the flipped indices).  

A word on tensor products: In this paper, the symbol `$\otimes$' means at least three things. When appearing between purely algebraic objects, such as algebras or just vector spaces, it is the usual, algebraic tensor product. Between $C^*$-algebras it always means the minimal, or injective tensor product (\cite[IV.4]{MR1873025}). Finally, on rare occasions, we use the so-called spatial tensor product (referred to as $W^*$-tensor product in \cite[IV.5]{MR1873025}) between von Neumann (or $W^*$) algebras. It will always be made clear what the nature of the tensored objects is, so that no confusion is likely to arise.


\subsection{Compact quantum groups}\label{subse.cstarqg}

This is by now a very rich and well-referenced theory, so we will be very brief, and will refer the reader to one of the many excellent sources (e.g. the papers and book cited below and the references therein) for details on the topic. 

No list of references would be complete without mentioning the seminal papers \cite{Woronowicz1987,MR943923}, where Woronowicz laid the foundation of the subject, introducing the main characters under the name ``compact matrix pseudogroups'', while an exposition of the main features of the theory is given by the same author in \cite{MR1616348}. Other good references are the survey paper \cite{MR1741102}, and \cite[11.4]{MR1492989}.  

As mentioned in the introduction, the main idea is that since one can study compact groups by means of the algebras of continuous functions on them, which are commutative, unital $C^*$-algebras with some additional structure, dropping the commutativity assumption but retaining the extra structure should still lead to interesting objects, which are trying to be ``continuous functions on a quantum group''. The additional structure just alluded to is captured in the following definition (\cite[3.1.1]{MR1741102}):

\begin{definition}\label{def.cstarqg}
A \define{compact quantum group} is a pair $(A,\Delta)$, where $A$ is a unital $C^*$-algebra, and $\Delta:A\to A\otimes A$ is a morphism of unital $C^*$-algebras satisfying the conditions
\begin{enumerate}
\item (Coassociativity) $(\Delta\otimes\id)\circ\Delta=(\id\otimes\Delta)\circ\Delta$;
\item (Antipode) The subspaces $\Delta(A)(1\otimes A)$ and $\Delta(A)(A\otimes 1)$ are dense in $A\otimes A$.
\end{enumerate}

$\cstarqg$ is the category whose objects are compact quantum groups, and whose morphisms $f:A\to B$ are unital $C^*$-algebra maps preserving the comultiplication in the sense that $(f\otimes f)\circ\Delta_A=\Delta_B\circ f$.
\end{definition}

The second condition needs some explanation. The space $\Delta(A)(1\otimes A)$ is defined as the linear span of products of the form $\Delta(a)(1\otimes b)\in A\otimes A$, and similarly for $\Delta(A)(A\otimes 1)$. The condition is named `antipode' because it has to do with the demand that $A$, regarded as a kind of bialgebra, have something like an antipode. For comparison, consider the case of an ordinary, purely algebraic bialgebra $B$. The condition that it have an antipode, i.e. that it be a Hopf algebra, is equivalent to the requirement that the map \[ B\otimes B\to B\otimes B,\quad x\otimes y\mapsto x_1\otimes x_2y\] be a bijection. 

Since compact quantum groups as defined above are morally functions of algebras, representations of the group must be comodules of some sort over the corresponding algebras (endowed with their comultiplication). Some care must be taken, as \Cref{def.cstarqg} makes no mention of a counit, and so the usual definition of comodule has to be modified slightly. The solution is (see \cite[discussion before Proposition 3.2.1]{MR1741102}): 

\begin{definition}\label{def.cstarqg_comod}
Let $(A,\Delta)$ be a compact quantum group. A \define{finite-dimensional comodule} over $A$ is a finite-dimensional complex vector space $V$ together with a coassociative coaction $\rho:V\to V\otimes A$ such that the corresponding element of $\End(V)\otimes A$ is invertible. 
\end{definition}

We will often drop the adjective `finite-dimensional'. \Cref{rem.corep} applies, and we will often refer to the coefficient comatrix of some basis as \define{being} the comodule structure. \cite[11.4.3, Lemma 45]{MR1492989} says that any comodule is \define{unitarizable}, in the sense that there is an inner product on $V$ making the coefficient matrix $u\in M_n(A)$ of an orthonormal basis unitary (cf. \Cref{def.unit}).


\subsection{CQG algebras}\label{subse.cqg}

These are the algebraic counterparts of compact quantum groups. More precisely, if a compact quantum group as in \Cref{def.cstarqg} plays the role of the algebra of continuous functions on a ``quantum group'', then the associated CQG algebra wants to be the algebra of representative functions of the quantum group, i.e. matrix coefficients of finite-dimensional unitary representations. The main reference for this subsection is \cite[11.1-4]{MR1492989}. 

Recall that a \define{Hopf $*$-algebra} $H$ is a Hopf algebra with a $*$-structure making $H$ into a $*$-algebra, and such that the comultiplication and counit are morphisms of $*$-algebras. This is the kind of structure that allows one to define what it means for a representations of a quantum group (i.e. a \define{co}module over the corresponding ``function algebra'') to be unitary. 

Let $V$ be an $n$-dimensional comodule over a Hopf $*$-algebra $H$.

\begin{definition}\label{def.unit}
If $(\ \mid\ )$ is an inner product on $V$, the pair $(V,(\ \mid\ ))$ is said to be \define{unitary} provided the coefficients $u_{ij}$ of an orthonormal basis $e_i$, $i=\overline{1,n}$ form a unitary matrix in $H$. 

A comodule $V$ is said to be \define{unitarizable} if there exists an inner product making it unitary. This is equivalent to saying that for any basis $(e_i)$, the coefficient matrix $(u_{ij})_{i,j}$ can be made unitary by conjugating it with a scalar $n\times n$ matrix. 
\end{definition}

This is \cite[Definition 5]{MR1492989}, and it is the correct compatibility condition for a comodule structure and an inner product. See also \cite[11.1.5, Proposition 11]{MR1492989} for alternative characterizations of unitary comodules. We are now ready to recall the main definition of this subsection (\cite[11.3.1, Definition 9]{MR1492989}): 

\begin{definition}\label{def.cqg}
A \define{CQG algebra} is a Hopf $*$-algebra which is the linear span of the coefficient matrices of its unitarizable (or equivalently, unitary) finite-dimensional comodules.  

The category having CQG algebras as objects and Hopf $*$-algebra morphisms as arrows will be denoted by $\cqg$. 
\end{definition}

\begin{remark}\label{rem.quotientCQG}
It is a simple but useful observation that a quotient Hopf $*$-algebra of a CQG algebra is automatically CQG. Indeed, a morphism of Hopf $*$-algebras will turn a unitary coefficient matrix into another such. 
\end{remark}

Let us recall that CQG algebras are automatically \define{cosemisimple} \cite[11.2]{MR1492989}, i.e. their categories of comodules are semisimple. Another way to say this is that a CQG algebra is the direct sum of its matrix subcoalgebras. 

The following example will play an important role in \Cref{se.cqgfinpres}. It is a family of ``universal'' CQG algebras, in a sense that will be made precise below (see \cite[11.3.1, Example 6]{MR1492989}, or \cite{MR1382726}, where these objects were introduced in their $C^*$-algebraic incarnation).  

\begin{example}\label{ex.auq}
Let $Q\in GL_n(\mathbb C)$ be a positive operator, and denote by $A_u(Q)$ the $*$-algebra freely generated by elements $u_{ij}$, $i,j=\overline{1,n}$ subject to the relations making both $u=(u_{ij})_{i,j}$ and $\displaystyle Q^{\frac 12} \overline u Q^{-\frac 12}$ unitary, where $\overline u=(u^*_{ij})_{i,j}$. Strictly speaking, the main character here is the pair $(A_u(Q),u)$ rather than just $A_u(Q)$: We always assume the $u_{ij}$ are fixed as part of the structure, and refer to them as \define{the standard generators} of $A_u(Q)$.  

One way to state the universality property mentioned above is: For any CQG algebra $A$ and any unitary coefficient matrix $v=(v_{ij})$ satisfying $S^2(v)=QvQ^{-1}$, the map $u_{ij}\mapsto v_{ij}$ lifts to a unique CQG algebra morphism $A_u(Q)\to A$. 

Note that $A_u(Q)$ has a standard $n$-dimensional unitary comodule with orthonormal basis $(e_i)_{i=1}^n$, with the obvious structure $e_j\mapsto \sum_i e_i\otimes u_{ij}$. 
\end{example}

There are various functors going back and forth between $\cqg$ and $\cstarqg$. First, since a CQG algebra is generated by elements of unitary matrices, there is, for any element of the algebra, a uniform bound on the norm that element can have when acting on any Hilbert space. It follows that any CQG algebra $A$ has an enveloping $C^*$-algebra $\overline{A}$. The fact that the comultiplication and counit of the CQG algebra lift to give $\overline{A}$ a compact quantum group structure follows from the universality property of this envelope, as does the functoriality of this construction (\cite[11.3.3]{MR1492989}). This functor will be denoted by $\univ:\cqg\to\cstarqg$. 

On the other hand, for any compact quantum group $B$, the coefficients of all comodules (\Cref{def.cstarqg_comod}) span a sub-Hopf $*$-algebra of $B$ in the obvious sense, and again, the construction is easily seen to be functorial. We denote this functor by $\alg:\cstarqg\to\cqg$. Moreover, $\alg(B)$ is the only dense sub-Hopf $*$-algebra of $B$ (\cite[A.1]{MR1862084}).    

Functors constructed in some natural way and going in opposite directions are in the habit of being adjoints, and this situation is no different: $\univ$ is the left adjoint. As it happens, $\alg$ almost has right adjoint too. `Almost' because only its restriction to the category of CQG algebras and one-to-one morphisms has a right adjoint, $\red$, associating to each CQG algebra $A$ the so-called \define{reduced} \cite[$\S$2]{MR1862084} compact quantum group having $A$ as its dense sub-Hopf $*$-algebra. It is the ``smallest'' such object, in the sense that any compact quantum group $A$ admits a unique surjective morphism $A\to \red(\alg(A))$ which restricts to the identity on $\alg(A)$. A detailed discussion on the interplay between the three functors mentioned in this paragraph can be found in \cite[6.2]{MR2805748}. 

It is probably clear by now that (the opposites of) $\cqg$ and $\cstarqg$ are ``the two reasonable choices for the category of compact quantum groups'' of the abstract. Which one is most convenient in any given case depends on which aspects of the theory one wishes to focus on. For representation-theoretic purposes, $\cqg$ seems to be the correct choice, since the CQG algebra $\alg(B)$ discussed above is tailor-made to capture all information about unitary $B$-comodules. On the other hand, there are purely analytic concepts (coamenability \cite{MR1862084}) whose very definition requires the use of $\cstarqg$. 

In this paper, the $\cqg$ vs. $\cstarqg$ distinction is a matter of technical necessity. For various reasons having to do with the topological aspect of being a $C^*$ rather than a $*$-algebra, $\cqg$ is the easier category to work with, as should be apparent from the announced results once we review the necessary category theory: The title of \Cref{se.cqgfinpres} (see \Cref{subse.finpres}) is stronger than that of \Cref{se.cstarqgsaft} (\Cref{subse.saft}).


\subsection{Locally compact and algebraic quantum spaces and semigroups}\label{subse.locallycompact}

One of the themes that will be explored in \Cref{se.appl} is, very roughly, the existence of ``compact quantum groups freely generated by quantum objects''. Here, `objects' can be things like `semigroups' or `spaces'. Keeping in mind that we are placing ourselves in the dual picture, where spaces are explored through functions on them, we recall in this subsection how non-unital $C^*$-algebras or plain $*$-algebras allow one to formalize such notions. 

A good, brief account of more or less everything we need for the locally compact side of the picture can be found in the `Notations and conventions' section of \cite{MR1832993} (assuming rudiments on multiplier algebras of $C^*$-algebras \cite[III.6]{MR1873025}). 

Recall that for not-necessarily-unital $C^*$-algebras $A$ and $B$, a \define{morphism} from $A$ to $B$ is by definition a continuous $*$-algebra homomorphism $f:A\to M(B)$ into the multiplier algebra of $B$ which is non-degenerate, meaning that the space $f(A)B$ is dense in $B$. It is then explained in \cite{MR1832993} how two such creatures can be composed, meaning that non-unital $C^*$-algebras together with morphisms as defined above constitute a category denoted here by $\cstaroh$. It is to be thought of as the category dual to that of locally compact quantum spaces. $\cstar$ is the subcategory consisting of unital $C^*$-algebras. Note that the non-degeneracy condition on morphisms automatically makes a $\cstaroh$ arrow between objects of $\cstar$ unital.  

For the definition of locally compact quantum semigroups we follow \cite{MR2210362}, referring again to the preliminary section of \cite{MR1832993} for the missing details on compositions of morphisms in $\cstaroh$ (needed to make sense of the coassociativity condition below).

\begin{definition}\label{def.cstarohqs}
A \define{locally compact quantum semigroup} is a pair $(A,\Delta)$, where $\Delta:A\to A\otimes A$ is a morphism in $\cstaroh$, coassociative in the obvious sense.  

The category $\cstarohqs$ has locally compact quantum semigroups as objects, and $\cstaroh$-morphisms compatible with comultiplications as arrows. 
\end{definition}

Turning now to the algebraic side, everything just said has a natural analogue. We again have to deal with multiplier algebras, this time of not-necessarily-unital $*$-algebras; \cite[Appendix]{MR1220906} provides sufficient background, and we will freely use the results and terminology therein. All $*$-algebras are assumed to be non-degenerate, in the sense that $ab=0$, $\forall b$ implies $a=0$ (the $*$-structure makes this condition symmetric). 

For $*$-algebras $A$ and $B$, a morphism $A\to B$ is by definition a $*$-homomorphism $f:A\to M(B)$, non-degenerate in the sense that $f(A)B$ spans $B$. Composition goes through essentially as in the $C^*$ case, and we thus get a category $\astaroh$. As before, the full subcategory $\astar$ on unital $*$-algebras has only unital morphisms as arrows. The algebraic counterpart to \Cref{def.cstarohqs} is

\begin{definition}\label{def.algstarohqs}
An \define{algebraic quantum semigroup} is a pair $(A,\Delta)$, where $\Delta:A\to A\otimes A$ is a coassociative morphism in $\astaroh$. 

The category $\astarohqs$ has algebraic quantum semigroups as objects and and $\astaroh$-morphisms compatible with comultiplications as arrows.
\end{definition}

\begin{remark}
Whatever results we prove below within the framework of \Cref{def.algstarohqs}, close analogues exist for plain complex algebras rather than $*$-algebras. I believe this one example is sufficient to illustrate how the universal constructions of \Cref{se.appl} go through in the algebraic, as well as the $C^*$-algebraic setting. 
\end{remark}

\subsection{SAFT categories and the adjoint functor theorem}\label{subse.saft}

As \Cref{se.appl} below is all about showing that certain functors have adjoints, in this subsection and the next we recall the categorical machinery involved in this. The main reference here is \cite{MR1712872}. The set of morphisms $x\to y$ in a category $\mathcal C$ will be denoted by $\mathcal C(x,y)$. Recall that categories with all (co)limits (always small in this paper) are said to be \define{(co)complete}, and functors preserving those (co)limits are called \define{(co)continuous} (so `complete' here means the same thing as Mac Lane's `small-complete' \cite[V]{MR1712872}). 

A class $S$ of objects in a category is said to be a \define{generator} (or a generating class) if any two distinct parallel arrows $f\ne g:y\to z$ stay distinct upon composition with an arrow $S\ni x\to y$ (\cite[V.7]{MR1712872}). We call category \define{generated} if there is a generating set (as opposed to a proper class). 

An arrow $f:x\to y$ in a category $\mathcal C$ is an \define{epimorphism} if arrows out of $y$ are uniquely determined by their ``restriction to $x$'' via composition with $f$ (\cite[I.5]{MR1712872}). The \define{quotient objects} of $x$ are the epimorphisms with source $x$, identified up to isomorphism in the comma category $x\downarrow\mathcal C$ of arrows with source $x$ \cite[II.6]{MR1712872}. Finally, $\mathcal C$ is said to be \define{co-wellpowered} if for every object $x$, the class of quotient objects of $x$ is actually a set.  

We explained above how the aim is to construct things like ``the compact quantum group freely generated by a quantum semigroup''. What this means, precisely, remembering that we are working with algebra-of-functions-type objects, is that we want a right adjoint to, say, the inclusion functor $\iota:\cqg\to\astarohqs$ (this is just one example; there is also a $C^*$ version). Typically, when trying to show that a functor $\iota$ is a left adjoint, one needs to check (1) that $\iota$ is cocontinuous (this is certainly necessary, as left adjoints are always cocontinuous) and (2) that some kind of solution set condition is satisfied \cite[V.6.2]{MR1712872}. For some categories, however, (2) is unnecessary: they are such that \define{any} cocontinuous functor out of them is automatically a left adjoint. One sufficient set of conditions that will ensure this is provided by the following result, due to Freyd and referred to in the literature as the \define{special adjoint functor theorem} (dual to \cite[V.8.2]{MR1712872}):

\begin{theorem}\label{th.adj}
Let $\mathcal C$ be a cocomplete, generated, and co-wellpowered category. Then, any cocontinuous functor with domain $\mathcal C$ is a left adjoint.\qedhere
\end{theorem}

In view of this result, it is natural to isolate the hypotheses:

\begin{definition}\label{def.saft}
A category is \define{SAFT} if it is cocomplete, generated, and co-wellpowered. 
\end{definition}

\begin{remark}\label{rem.saftcomplete}
Not-necessarily-cocomplete categories satisfying the adjoint functor theorem in the sense that functors are only required to preserve those colimits which exist are called `compact' in \cite{MR850527}. One important property of compact (and hence SAFT) categories is that they are automatically complete. This is, for example, the implication (ii) $\Rightarrow$ (v) in \cite[Theorem 5.6]{MR850527}.  
\end{remark}


\subsection{Finitely presentable categories}\label{subse.finpres}

One way to be SAFT is to be what in the literature is called `locally presentable'. We review the main features of the theory here, and refer mainly to \cite{MR1294136} for details. For brevity, in this paper we drop the word `locally'. 

A poset $(J,\le)$ is said to be \define{filtered} if every finite subset is majorized by some element (this is \cite[1.4]{MR1294136}, restricted to posets as opposed to arbitrary categories). In the sequel, a \define{filtered diagram} in a category $\mathcal C$ is a functor a $(J,\le)\to\mathcal C$, and a \define{filtered colimit} is a colimit of such a functor. Then, \cite[1.9]{MR1294136} is (essentially, via \cite[1.5]{MR1294136}):

\begin{definition}\label{def.finpres}
An object $x\in\mathcal C$ is \define{finitely presentable} if the functor $\mathcal C(x,-)$ preserves filtered colimits. 

The category $\mathcal C$ is finitely presentable if it is cocomplete, and there is a set $S$ of finitely presentable objects such that every object in $\mathcal C$ is a filtered colimit of objects in $S$.  
\end{definition}

More rigorously, the last condition says that every object is the colimit of a functor $F:J\to\mathcal C$ taking values in $S$, with $(J,\le)$ filtered.

\begin{example}
All categories familiar from algebra, of the form `set with this or that kind of structure', such as groups, abelian groups, monoids, semigroups, algebras, $*$-algebras, modules over a ring, etc. are finitely presentable. These are the so-called finitary varieties of algebras \cite[3.A]{MR1294136}, `finitary' having to do with `\define{finitely} presentable'. 

The terminology is also inspired by such examples. In the category of modules over a ring, say, an object is finitely presentable in the above abstract sense if and only if it has a finite presentation in the usual sense (in full generality, the result is \cite[3.11]{MR1294136}). 
\end{example}

What matters here is that as mentioned above, finitely presentable implies SAFT. Indeed, cocompleteness is part of \Cref{def.finpres}, and the generating set required for SAFT-ness almost is: $S$ is easily seen to be a generator. Co-wellpowered-ness, on the other hand, is the difficult result \cite[1.58]{MR1294136}. 

The following proposition is the criterion of finite presentability we use in the proof of the main result of \Cref{se.cqgfinpres}. It is a consequence of \cite[1.11]{MR1294136} (via 0.5, 0.6 of op. cit.), and in order to state it, one more piece of terminology is needed. 

\begin{definition}\label{def.reggen}
We will say that a generator $S$ of a cocomplete category $\mathcal C$ is \define{regular} if every object of $\mathcal C$ is the coequalizer of two parallel arrows $f,g:y\to z$, where $y$ and $z$ are coproducts of objects in $S$. 
\end{definition}

\begin{proposition}\label{prop.finpres}
A cocomplete category with a regular generator consisting of finitely presentable objects is finitely presentable. 
\qedhere
\end{proposition}


\section{The category of CQG algebras is finitely presentable}\label{se.cqgfinpres}

This section is devoted to proving the result in the title:

\begin{theorem}\label{th.cqgfinpres}
The category $\cqg$ is finitely presentable.
\end{theorem}

The main tool in the proof is \Cref{prop.finpres}, according to which cocompleteness (by now well known) is first on the agenda.

\begin{proposition}\label{prop.cqgcoco}
$\cqg$ is cocomplete. 
\end{proposition} 
\begin{proof}
It is enough to show that the category has coproducts and coequalizers of parallel pairs of arrows \cite[V.2.1]{MR1712872}. Both will be constructed as simply the colimits of the underlying diagrams of $*$-algebras (i.e. coalgebra structures play no role in the construction of colimits). 

Coproducts are essentially constructed in \cite[Theorem 1.1]{MR1316765}. That result is concerned with the $C^*$-algebraic version (constructing coproducts in the category $\cstarqg$), but as remarked by Wang at the end of \cite[$\S$1]{MR1316765}, the algebraic version holds as well. Given a set of CQG agebras, the universal property of the coproduct of underlying algebras gives this coproduct the extra structure that will make it into a CQG algebra; we leave the details to the reader.    

To construct coequalizers, let $f,g:A\to B$ be morphisms of CQG algebras. The ideal $I$ generated by the elements $f(a)-g(a)$, $a\in A$ is in fact a coideal, as well as invariant under $*$. The latter assertion is trivial, so let us focus on $I$ being a coideal. Compatibiity of $f$ and $g$ with counits says that $f(a)-g(a)$ is annihilated by $\varepsilon_B$. On the other hand, the familiar computation \begin{alignat*}{-1} \Delta_B(f(a)-g(a)) =&\ (f\otimes f)(\Delta_A(a)) - (g\otimes g)(\Delta_A(a))\\ =&\ f(a_1) \otimes f(a_2) - g(a_1) \otimes g(a_2) \\  =&\ f(a_1) \otimes (f-g)(a_2) + (f-g)(a_1)\otimes g(a_2)\in B\otimes I+I\otimes B \end{alignat*} shows that $I$ plays well with the comultiplication. It follows that the coequalizer of $f$ and $g$ in $\astar$ is a quotient Hopf $*$-algebra of $B$, and hence a CQG algebra by \Cref{rem.quotientCQG}.
\end{proof}

The plan now is to show that the set $S$ consisting of the CQG algebras $A_u(Q)$ of \Cref{ex.auq} (for all possible positive operators $Q$, of all possible sizes) satisfies the hypotheses of \Cref{prop.finpres}: Every $A_u(Q)$ is finitely presentable in $\cqg$ in the sense of \Cref{def.finpres}, and $S$ is a regular generator. We start with the former.

\begin{proposition}\label{prop.AuQfinpres}
For any positive $Q\in GL_n(\mathbb C)$, the object $A=A_u(Q)\in\cqg$ is finitely presentable.  
\end{proposition} 
\begin{proof}
Let $(J,\le)$ be a filtered poset, and let $A_j$, $j\in J$ implement a functor $J\to\cqg$ by means of CQG algebra morphisms $\iota_{j'j}:A_j\to A_{j'}$ for $j\le j'$. Denote also by $\displaystyle \iota_i:A_i\to B=\varinjlim_jA_j$ the structural morphisms into the colimit. We have to show that the canonical map 
\begin{equation}\label{eq.Aufinpres} 
	\varinjlim_j\cqg(A,A_j) \to \cqg(A,B) 
\end{equation} 
is a bijection. 

It is clear from the description in \Cref{ex.auq} that $A$ is finitely presented as a $*$-algebra, and is hence a finitely presentable object in $\astar$ (\cite[3.11]{MR1294136} with $\lambda=\aleph_0$). In conclusion, the canonical map 
\begin{equation}\label{eq.algpres} 
	\varinjlim_j\astar(A,A_j)\to\astar(A,B) 
\end{equation} 
is bijective, and the injectivity of \Cref{eq.Aufinpres} follows from the commutative square
\[
 \begin{tikzpicture}[anchor=base]
   \path (0,0) node (1) {$\varinjlim\cqg(A,A_j)$} +(4,0) node (2) {$\cqg(A,B)$} +(0,-2) node (4) {$\varinjlim\astar(A,A_j)$} +(4,-2) node (3) {$\astar(A,B)$,}; 
   \draw[->] (1) -- (2);
   \draw[->] (2) -- (3);
   \draw[->] (1) -- (4);
   \draw[->] (4) -- (3);
 \end{tikzpicture}
\]
where the vertical arrows are the obvious inclusions. 

To prove that \Cref{eq.Aufinpres} is surjective, fix a CQG algebra morphism $f:A\to B=\varinjlim A_j$. It is, in particular, a morphism in $\astar$, so by the surjectivity of \Cref{eq.algpres}, it factors through a unital $*$-algebra morphism $f_i:A\to A_i$ for some $i\in J$, and hence through $f_j=\iota_{ji}\circ f_i:A\to A_j$ for $j\ge i$. For such $j$, consider the diagram
\[
	\begin{tikzpicture}[anchor=base]
   	\path (0,0) node (1) {$A$} +(3,0) node (2) {$A_j$} +(6,0) node (3) {$B$} +(6,-2) node (4) {$B\otimes B$,} +(3,-2) node (5) {$A_j\otimes A_j$} +(0,-2) node (6) {$A\otimes A$}; 
   	\draw[->] (1) -- (2) node[pos=.5,auto] {$\scriptstyle f_j$};
   	\draw[->] (2) -- (3) node[pos=.5,auto] {$\scriptstyle \iota_j$};
   	\draw[->] (6) -- (5) node[pos=.5,auto] {$\scriptstyle f_j\otimes f_j$};
   	\draw[->] (5) -- (4) node[pos=.5,auto] {$\scriptstyle \iota_j\otimes\iota_j$};
   	\draw[->] (1) -- (6);
   	\draw[->] (2) -- (5);
   	\draw[->] (3) -- (4);
  \end{tikzpicture}
\]
where the vertical maps are comultiplications. The commutativity of the outer rectangle is nothing but the preservation of coproducts by $f=\iota_j\circ f_j$, while the right hand square commutes because $\iota_j:A_j\to B$ is by definition a colimit in $\cqg$. It follows that the two $J$-indexed systems of morphisms $\Delta_{A_j}\circ f_j$ and $(f_j\otimes f_j)\circ\Delta_A$ become equal upon composing further with $\iota_j\otimes\iota_j$. The fact that they must then be equal for sufficiently large $j$ follows from the next lemma, which says essentially that $\iota_j\otimes\iota_j$ make $B\otimes B$ the colimit of the diagram consisting of the maps $\iota_{j'j}\otimes\iota_{j'j}$, together with the injectivity of \[ \varinjlim_j \astar(A,A_j\otimes A_j)\longrightarrow \astar(A,\varinjlim_j A_j\otimes A_j) \] resulting from the finite presentability of $A$ in $\astar$.    
\end{proof}

\begin{lemma}\label{le.tensorsquare}
The tensor square endofunctor $A\mapsto A\otimes A$ on $\astar$ preserves filtered colimits. 
\end{lemma}
\begin{proof}
Working with $*$-algebras is of little importance here: The forgetful functor from any finitary variety of algebras to the category $\cat{Set}$ of sets creates filtered colimits. We refer again to \cite[3.A]{MR1294136} for background on varieties of algebras. The claim just made can be proven either by realizing the variety as the Eilenberg-Moore category of a finitary monad (i.e. one which preserves filtered colimits) on $\cat{Set}$ \cite[3.18]{MR1294136}, or directly. It follows that it is enough to prove the analogous statement in the category $\cat{Vec}$ of complex vector spaces. 

Let $(J,\le)$ be a filtered poset, $\iota_{j'j}:V_j\to V_{j'}$ a functor from it to $\cat{Vec}$, and $\displaystyle V=\varinjlim_j V_j$. We have to show that the canonical map $\displaystyle \varinjlim_j (V_j\otimes V_j)\to V\otimes V$ is an isomorphism. 

Let $(J\times J,\le)$ be the cartesian square of the category $(J,\le)$; it is simply the poset structure on the set $J\times J$ defined by $(i,j)\le (i',j')$ iff $i\le i'$ and $j\le j'$. Consider the functor $F:(J\times J,\le)\to\cat{Vec}$ given by $(i,j)\mapsto V_i\otimes V_j$ (with the obvious action on morphisms). For any vector space $W$, the endofunctor $W\otimes\bullet:\cat{Vec}\to\cat{Vec}$ is left adjoint to $\Hom(W,\bullet)$, and hence cocontinuous. Applying this observation first to $V_i$ and then to $V$, we get the last two isomorphisms in the chain \begin{equation}\label{eq.VotimesV} 
	\varinjlim_{i,j}(V_i\otimes V_j)\cong \varinjlim_i\varinjlim_j(V_i\otimes V_j)\cong \varinjlim_i (V_i\otimes V)\cong V\otimes V. 
\end{equation} 
The first one, on the other hand, is the usual Fubini-type separation of variables for colimits  (\cite[IX.8]{MR1712872}).
  
The original poset $J$ sits diagonally inside $J\times J$ as the set of pairs $(j,j)$. Moreover, the fact that $J$ is filtered translates to $J$ being \define{cofinal} in $J\times J$ in the sense that everyone in the latter is majorized by someone in the former. But it then follows \cite[IX.3.1]{MR1712872} that the canonical map \[ \varinjlim_j (V_j\otimes V_j)=\varinjlim F|_J\longrightarrow \varinjlim F = \varinjlim_{i,j}(V_i\otimes V_j) \] is an isomorphism. Composing it with \Cref{eq.VotimesV} finishes the proof.  
\end{proof}

\begin{remark}
We mentioned briefly at the end of \Cref{subse.cqg} that the category $\cstarqg$ is not as friendly as $\cqg$. \Cref{le.tensorsquare}, for example, is somewhat problematic. The problem with the above proof is that it hinges on functors of the form $A\otimes\bullet:\astar\to\astar$ preserving filtered colimits; I do not know whether the analogous result holds for the category $\cstar$ of unital $C^*$-algebras with the minimal tensor product (which is what would be needed to make the proof work verbatim for $\cstarqg$). 

Some partial results (which we do not prove here) are that (a) for a $C^*$-algebra $A$, the functor $A\otimes\bullet$ does preserve filtered colimits of injections, (b) the same functor preserves all filtered colimits provided $A$ is an \define{exact} $C^*$-algebra in the sense of \cite{MR1271145} (this simply means that minimal tensoring with $A$ preserves short exactness of sequences in $\cstaroh$), and (c) the \define{maximal} tensor product with $A$ does preserve all filtered colimits. This suggests that trying to adapt \Cref{le.tensorsquare} to $\cstar$ adds a layer of difficulty, in that one has to deal with issues like nuclearity and exactness.    
\end{remark}

The last piece of the puzzle is

\begin{proposition}\label{prop.reggen}
The set $S$ of all $A_u(Q)$ is a regular generator in $\cqg$. 
\end{proposition}
\begin{proof}
According to \Cref{def.reggen}, we have to show that an arbitrary CQG algebra $A$ is the coequalizer of two arrows $f,g:Y\to Z$ between coproducts of $A_u(Q)$'s. 

Let $V^\alpha$ be representatives for the set $\widehat A$ of unitary simple comodules of $A$. Then, $A$ is the direct sum of the matrix coalgebras $C^\alpha$ spanned by the unitary coefficient matrices $v^\alpha=(v^\alpha_{ij})$ with respect to orthonormal bases $(e^\alpha_i)$ of $V^\alpha$.

The squared antipode $S^2$ conjugates every matrix $v^\alpha$ by some positive operator $Q^\alpha$ \cite[11.2.3, Lemma 30]{MR1492989}, and hence, by the universality property of the $A_u(Q)$'s as cited in \Cref{ex.auq}, the assignment $u_{ij}\mapsto v^\alpha_{ij}$ defines a CQG algebra morphisms $A_u(Q^\alpha)\to A$. Moreover, the resulting map 
\begin{equation}\label{eq.reggen1}
	\pi:\coprod_{\widehat A} A_u(Q^\alpha)\to A
\end{equation} 
is surjective. The left hand side of this expression will be our $Z$. For $\alpha\in\widehat A$, we denote by $u^\alpha$ the standard coefficient matrix in $A_u(Q^\alpha)$ (earlier in this paragraph, where we reasoned one $\alpha$ at a time, it was denoted simply by $u$).   

The two morphisms $f,g:Y\to Z$ that we are looking for (and whose coequalizer \Cref{eq.reggen1} should be) ought to somehow recover the relations of $A$, i.e. the multiplication table with respect to the basis $(v^\alpha_{ij})$ of $A$. 

To define $f$, fix $\alpha,\beta\in \widehat A$. A simple calculation shows that $u^\alpha u^\beta=(u^\alpha_{ij}u^\beta_{kl})_{ik,jl}$ is a unitary coefficient matrix of $Z=\coprod A_u(Q^\alpha)$, which the squared antipode of this CQG algebra conjugates by $Q^{\alpha\beta}=Q^\alpha\otimes Q^\beta$ (this is just notation). It follows that there is a unique morphism $A_u(Q^{\alpha\beta})\to Z$ defined by $u_{ik,jl}\mapsto u^\alpha_{ij}u^\beta_{kl}$. Putting all of these together for all pairs of comodules, we get
\begin{equation*}\label{eq.reggen2}
	f:Y=\coprod_{\widehat A\times\widehat A} A_u(Q^{\alpha\beta})\to \coprod A_u(Q^\alpha)=Z.
\end{equation*} 
As before, since we need to distinguish between the various coefficient matrices in $Y$, we denote them by $u^{\alpha\beta}$ in the obvious way. 

We now start on our way towards constructing $g:Y\to Z$. The same game as in the previous paragraph can be played in $A$: For any $\alpha,\beta\in\widehat A$, $v^\alpha v^\beta=(v^\alpha_{ij}v^\beta_{kl})_{ik,jl}$ is the unitary coefficient matrix with respect to the tensor product basis $e^\alpha_i\otimes e^\beta_k$ of the tensor product Hilbert space $V^\alpha\otimes V^\beta$. But now, since we are in $A$, the elements of this matrix can be expressed as linear combinations of $v^\gamma_{ij}$'s. In order to avoid cumbersome indices on the coefficients of such linear combinations, we simply write \[ v^\alpha_{ij} v^\beta_{kl}=\ell^{\alpha\beta}_{ik,jl}\in\bigoplus_\gamma C^\gamma, \] where $C^\gamma$ is the matrix coalgebra corresponding to $\gamma\in\widehat{A}$, and $\gamma$ ranges over the simple comodules appearing in the decomposition of $V^\alpha\otimes V^\beta$. 

Now, because the restriction of $\pi$ to the direct sum $C\le Z$ of matrix coalgebras spanned by $u^\alpha\subset Z$ is one-to-one (in fact, this restriction is by definition an isomorphism onto $A$), the elements $\ell^{\alpha\beta}_{ik,jl}$ defined above lift uniquely to elements of $C$, and we slightly abusively denote these lifts by $\pi^{-1}(\ell^{\alpha\beta}_{ik,jl})$. I claim that for fixed $\alpha$ and $\beta$, these elements form a unitary coefficient matrix which $S^2$ conjugates by $Q^\alpha\otimes Q^\beta$. Indeed, all of these properties can be stated inside $C$ (without appealing to multiplication), using only the antipode and the $*$-structure (being unitary, for example, amounts to the antipode turning $\pi^{-1}(\ell^{\alpha\beta}_{ik,jl})$ into $\pi^{-1}(\ell^{\alpha\beta}_{jl,ik})^*$); since $\pi$ preserves both the antipode and the $*$ structure and its restriction to $C$ is a coalgebra isomorphism, the properties all lift from the $\ell$'s to the $\pi^{-1}(\ell)$'s. 

Finally, the claim just proven allows us to construct $g:Y\to Z$ by sending $u^{\alpha\beta}_{ik,jl}$ to $\pi^{-1}(\ell^{\alpha\beta}_{ik,jl})$. The coequalizer of $f$ and $g$ is the quotient of $Z$ by the relations imposing on $u^\alpha_{ij}$ the same multiplication table as that of the $v^\alpha_{ij}$'s, so it is now clear that this coequalizer is precisely $\pi:Z\to A$.  
\end{proof}

We can now put the last few results together:

\begin{proof of cqgfinpres}
We know from \Cref{prop.cqgcoco} that $\cqg$ is cocomplete, and from \Cref{prop.AuQfinpres,prop.reggen} that a set of finitely presentable objects forms a regular generator. The conclusion follows from \Cref{prop.finpres}. 
\end{proof of cqgfinpres}

\begin{remark}\label{rem.cqgab}
In essentially the same way, we can show that the category $\cqgab$ of commutative CQG algebras is finitely presentable. In this case, all distinctions between the algebraic and the $C^*$-algbraic vanish: The restriction of $\univ$ to $\cqgab$ is an equivalence onto the full subcategory $\cstarqgab$ of $\cstarqg$ consisting of commutative algebras. Moreover, $\cqgab$ (or $\cstarqgab$) is nothing but the opposite of the category of compact groups, with a compact group $G$ corresponding to the CQG algebra of representative functions on it. 

The only changes we need to make to the proofs in order to adapt the presentability result to $\cqgab$ are (a) substitute tensor products (of perhaps infinite families) for coproducts, and (b) use the set of CQG algebras associated to all unitary groups $U_n$ for a generator, instead of the $A_u(Q)$'s.  
\end{remark}

\begin{remark}
Although, strictly speaking, SAFT-ness would have sufficed for the purposes of \Cref{se.appl}, the finite presentability of $\cqg$ is interesting in its own right, as it is somewhat surprising: Given the close relationship between $\cqg$ and $\cstarqg$, discussed a little in \Cref{subse.cqg} above, one might think that the former category should look more or less like ``unital $C^*$-algebras with a lot of extra structure'', and hence should be at least as reluctant to be finitely presentable as the category $\cstar$ of unital $C^*$-algebras. However, this is not the case. 

There is a more general notion of presentability for categories (local presentability in the literature, e.g. \cite{MR1294136}) parametrized by a regular cardinal number, so that the technical term for `finitely presentable' is `$\aleph_0$-presentable'; the larger the cardinal, the weaker the notion. Now, it can be shown that $\cstar$ is $\aleph_1$-presentable but \define{not} finitely presentable. Worse still, the same is true in the commutative setting: Although the previous remark notes that $\cqgab$ is finitely presentable, the category of commutative unital $C^*$-algebras is $\aleph_1$, but not finitely presentable. 
\end{remark}


\section{\texorpdfstring{$\cstarqg$}{C*QG} is SAFT}\label{se.cstarqgsaft}

The main result of the section is the one just stated:

\begin{theorem}\label{th.cstarqgsaft}
The cateory $\cstarqg$ is SAFT. 
\end{theorem}

We prove the three properties required for SAFT-ness (\Cref{def.saft}) separately.

\begin{proposition}\label{prop.cstarqgcoco}
$\cstarqg$ is cocomplete. 
\end{proposition}
\begin{proof}
This parallels the proof of \Cref{prop.cqgcoco} by constructing coequalizers and coproducts, so we will be brief. 

As noted in the proof just mentioned, coproducts are constructed in \cite[Theorem 1.1]{MR1316765}. 

As for coequalizers, they are constructed as before. The coequalizer of two morphisms $f,g:A\to B$ in $\cstarqg$ is the quotient of $B$ by the closed ideal $I$ generated by $f(a)-g(a)$, and the argument from the proof of \Cref{prop.cqgcoco} can be paraphrased to show this. 

Although Sweedler notation is not available anymore (because we are working with $C^*$ tensor products rather than algebraic ones), the computation carried out there can be written down in a Sweedler-notation-free manner as saying that $\Delta_B\circ(f-g)$ equals $(f\otimes(f-g)+(f-g)\otimes g)\circ\Delta_A$. It follows that $B/I$ inherits a coassociative comultiplication and $B\to B/I$ respects it, while the (Antipode) condition of \Cref{def.cstarqg} follows immediately from that of $B$. In conclusion, the quotient $B\to B/I$ is naturally a map in $\cstarqg$. 
\end{proof}

\begin{remark}
Filtered colimits and pushouts in $\cstarqg$ are constructed in \cite[3.1,3.4]{MR1316765} in the case when the morphisms in the diagram are one-to-one. According to (the proof of) \Cref{prop.cstarqgcoco}, injectivity is not necessary in order to conclude that the colimit of a diagram in $\cstarqg$ in the category of unital $C^*$-algebras is automatically endowed with a compact quantum group structure. 
\end{remark}

Next in line is the generation condition of \Cref{def.saft}. It turns out that $A_u(Q)$ will once more come in handy. We need them as objects of $\cstarqg$, so recall the enveloping $C^*$-algebra functor $\univ:\cqg\to\cstarqg$.

\begin{proposition}\label{prop.cstarqggen}
The set $\univ(A_u(Q))$ for $Q$ ranging over all positive matrices generates $\cstarqg$. 
\end{proposition}
\begin{proof}
That the $A_u(Q)$ form a generator in $\cqg$ is part of the statement of \Cref{th.cqgfinpres}. It is a simple exercise that left adjoints, such as $\univ$, turn generators into generators provided their right adjoints are faithful. In our case, the faithfulness of the right adjoint $\alg$ to $\univ$ follows from the density of the inclusion $\alg(A)\subset A$: Any arrow $f:A\to B$ in $\cstarqg$ is the extension by continuity of $\alg(f):\alg(A)\to\alg(B)$, and hence $\alg(f)=\alg(g)$ implies $f=g$. 
\end{proof}

The only ingredient of \Cref{def.saft} still to be addressed is co-wellpoweredness. Recall (\Cref{subse.saft}) that this meant that every object has only a set of quotient objects. It will help, then, to know exactly which morphisms in $\cstarqg$ are epimorphisms; this is what the following result does. 

\begin{proposition}\label{prop.epi}
A morphism $f:A\to B$ in $\cstarqg$ is an epimorphism if and only if it is surjective.
\end{proposition}


\begin{proof}
As usual in categories where objects are sets with some kind of structure and morphisms are maps preserving that structure, the implication surjective $\Rightarrow$ epimorphism is immediate. 

To prove the other implication, we will show that if $f$ is an epimorphism, then $\alg(f)$ is surjective (the conclusion follows from the denseness of $\alg(B)\subset B$). Since we can always substitute the image of $f$ for $A$, we can (and will) assume that $f$ is injective. 

First, recall the construction $\cqg\ni X\mapsto\red(X)\in\cstarqg$ mentioned in \Cref{subse.cqg}. It is functorial when restricted to the category $\cqg_{\mathrm{inj}}$ of CQG algebras and injective morphisms, (this is the essence of \cite[6.2.4]{MR2805748}). In order to keep the notation manageable, indicate the functors $\alg$ and $\red$ by superscripts, as in $X^a$ for $\alg(X)$, $X^r$ for $\red(X)$, $X^{ar}$ for $\red(\alg(X))$, etc.; the same conventions are in place for morphisms.   

Let $\displaystyle \iota:B\to B\coprod_AB$ be the left hand canonical inclusion into the pushout of $f$ along itself in the category $\cstarqg$, or equivalently (by the proof of \Cref{prop.cstarqgcoco}), in the category $\cstar$ of unital $C^*$-algebras. One condition equivalent to $f$ being epimorphic is that $\iota$ be an isomorphism. Similarly, we denote by $\iota'$ the left hand inclusion $\displaystyle B^a\to B^a\coprod_{A^a}B^a$ into the pushout in $\astar$. Note that $\iota$ and $\iota'$ are both injective, as they have left inverses by the universality property of pushouts.

Assume now that $f^a$ is not surjective. Then, I claim that (a) $\iota'$ cannot be surjective (equivalently, an isomorphism), and hence (b) neither can $(\iota')^r$. That (a) does indeed imply (b) follows from the fact \cite[6.2.12]{MR2805748} that the functor $\red:\cqg_{\mathrm{inj}}$ reflects isomorphisms. 

To prove (a), recall that an inclusion $K\subseteq H$ of cosemisimple Hopf algebras always splits as a $K$-$K$-bimodule map (e.g. as argued in the proof of \cite[2.0.4]{2011arXiv1110.6701C}). Applied to the inclusion $f^a:A^a\to B^a$, this observation yields a direct sum decomposition $B^a=A^a\oplus M$ as $A^a$-$A^a$-bimodules for some non-zero $M$, and the pushout $\displaystyle B^a\coprod_{A^a}B^a$ breaks up as a direct sum of $2^n$ copies of $M^{\otimes n}$ for $n\ge 0$ (tensor product of $A^a$-$A^a$-bimodules). Moreover, $\iota'$ identifies $B^a$ with the summand $A^a\oplus M$ therein. 

Now consider the commutative diagram  
\[
 \begin{tikzpicture}[anchor=base]
  \path (0,0) node (1) {$B$} +(2,0) node (2) {$B\coprod_AB$} +(5,0) node (3) {$B^{ar}\coprod_{A^{ar}}B^{ar}$} +(0,-2) node (4) {$B^{ar}$} +(5,-2) node (5) {$(B^a\coprod_{A^a}B^a)^r$}; 
  \draw[->] (1) -- (2) node[pos=.5,auto] {$\scriptstyle \iota$};
  \draw[->] (2) -- (3);
  \draw[->] (1) -- (4);
  \draw[->] (3) -- (5);
  \draw[->] (4) -- (5) node[pos=.5,auto] {$\scriptstyle (\iota')^r$};
 \end{tikzpicture}
\]
where the right hand vertical map comes from the universality property of the pushout applied to the two inclusions $\displaystyle B^a\to B^a\coprod_{A^a}B^a$, and the other two unnamed maps are induced by the surjection $B\to B^{ar}$. We have just argued that if $f^a$ is not surjective, then the lower left corner path is not surjective. But then the upper right corner path isn't either. However, the right hand upper horizontal arrow is surjective, as is the right hand vertical arrow. In conclusion, the only morphism in this path which can fail to be surjective (under the assumption that $f^a$ is not surjective) is $\iota$. 
\end{proof}

\begin{remark}\label{rem.epi}
The proof makes it clear that the analogous result is true for $\cqg$, i.e. epimorphisms are surjective. 
\end{remark}

Since for any compact quantum group $A$ there is only a set of quotients of $\alg(A)$ and hence only a set of compact quantum groups having such quotients as dense CQG subalgebras, the previous result implies:

\begin{proposition}\label{prop.cstarqgcowell}
$\cstarqg$ is co-wellpowered.\qedhere 
\end{proposition}

\begin{proof of cstarqgsaft}
\Cref{prop.cstarqgcoco,prop.cstarqggen,prop.cstarqgcowell} are precisely what is required by \Cref{def.saft}. 
\end{proof of cstarqgsaft}


\section{Applications}\label{se.appl}

The goal of this section is to apply \Cref{th.cqgfinpres,th.cstarqgsaft}, together with the adjoint functor theorem and abstract properties of presentable or SAFT categories, to the construction of compact quantum groups with various universal properties. These constructions fall roughly into two categories: right adjoints to functors defined on $\cqg$ or $\cstarqg$, as direct applications of \Cref{th.adj}, and left adjoints arising in a slightly more roundabout way in \Cref{subse.kac}. 

Note that limits in the categories $\cqg$ and $\cstarqg$, discussed in the next subsection, fit in this framework as right adjoints: Given a small category $J$ and a category $\mathcal C$, the limits of functors $J\to\mathcal C$, if they exist, can be obtained as images of a right adjoint to the diagonal functor $\Delta:\mathcal C\to\mathcal C^J$ (the latter is notation for the category of all functors $J\to\mathcal C$) sending an object $c\in\mathcal C$ to the functor $\Delta(c):J\to\mathcal C$ constant at $c$.


\subsection{Limits in \texorpdfstring{$\cqg$}{CQG} and \texorpdfstring{$\cstarqg$}{C*QG}}\label{subse.lim}

We now know from \Cref{th.cqgfinpres,th.cstarqgsaft} that both $\cqg$ and $\cstarqg$ are SAFT. Remembering that SAFT-ness implies completeness (\Cref{rem.saftcomplete}), we get:

\begin{theorem}\label{th.lim}
The categories $\cqg$ and $\cstarqg$ are complete.\qedhere 
\end{theorem}

Limits in these categories are quantum analogues of \define{co}limits of compact groups. It is natural to ask whether functors $J\to\cqg$ or $\cstarqg$ taking commutative values have commutative limits, or in other words, whether \Cref{th.lim} recovers ordinary colimits of compact groups. Since $\cqgab$ is complete (by \Cref{rem.cqgab} or simply constructing coequalizers and coproducts in the category of compact groups), the next result confirms this:

\begin{proposition}\label{prop.limcomm}
The inclusions $\cqgab\to\cqg$ and $\cstarqgab\to\cstarqg$ are right adjoints. 
\end{proposition}
\begin{proof}
In both cases the left adjoint is \define{abelianization}, associating to an object $A\in\cqg$ (or $\cstarqg$) its largest commutative quotient $*$-algebra (resp. $C^*$-algebra) $A_{ab}$. This is all rather routine, so we omit most of the details.

The universality property of the canonical quotient map $\pi:A\to A_{ab}$ ensures that the composition 
\[
 \begin{tikzpicture}[anchor=base]
   \path (0,0) node (1) {$A$} +(2,0) node (2) {$A\otimes A$} +(5,0) node (3) {$A_{ab}\otimes A_{ab}$}; 
   \draw[->] (1) -- (2) node[pos=.5,auto] {$\scriptstyle \Delta$};
   \draw[->] (2) -- (3) node[pos=.5,auto] {$\scriptstyle \pi\otimes\pi$};
 \end{tikzpicture}
\]
factors through a map $\Delta_{ab}:A_{ab}\to A_{ab}\otimes A_{ab}$. The coassociativity of $\Delta_{ab}$ follows from the uniqueness of the factorization of $\pi^{\otimes 3}\circ(\Delta\otimes\id)\circ\Delta:A\to A_{ab}^{\otimes 3}$ through $\pi$. By construction, we get a commutative square 
\[
 \begin{tikzpicture}[anchor=base]
   \path (0,0) node (1) {$A$} +(3,0) node (2) {$A_{ab}$} +(0,-2) node (3) {$A\otimes A$} +(3,-2) node (4) {$A_{ab}\otimes A_{ab}$}; 
   \draw[->] (1) -- (2) node[pos=.5,auto] {$\scriptstyle \pi$};
   \draw[->] (3) -- (4) node[pos=.5,auto] {$\scriptstyle \pi\otimes\pi$};
   \draw[->] (1) -- (3) node[pos=.5,auto,swap] {$\scriptstyle \Delta$};
   \draw[->] (2) -- (4) node[pos=.5,auto] {$\scriptstyle \Delta_{ab}$};
 \end{tikzpicture}
\]  
The rest of the structure and properties (e.g. counit $\varepsilon_{ab}:A_{ab}\to\mathbb C$ and counitality of $\Delta_{ab}$ in the algebraic case) follow similarly, as do the functoriality and the desired universality property of $A\mapsto A_{ab}$. To see the latter, for example, let $f:A\to B$ be a morphism in $\cqg$ or $\cstarqg$ with $B$ commutative. Then, $f$ factors uniquely as $f_{ab}\circ\pi$ for an algebra map $f_{ab}:A_{ab}\to B$. The outer rectangle of  
\[
	\begin{tikzpicture}[anchor=base]
   	\path (0,0) node (1) {$A$} +(3,0) node (2) {$A_{ab}$} +(6,0) node (3) {$B$} +(0,-2) node (4) {$A\otimes A$} +(3,-2) node (5) {$A_{ab}\otimes A_{ab}$} +(6,-2) node (6) {$B\otimes B$}; 
   	\draw[->] (1) -- (2) node[pos=.5,auto] {$\scriptstyle \pi$};
   	\draw[->] (2) -- (3) node[pos=.5,auto] {$\scriptstyle f_{ab}$};
   	\draw[->] (4) -- (5) node[pos=.5,auto] {$\scriptstyle \pi\otimes\pi$};
   	\draw[->] (5) -- (6) node[pos=.5,auto] {$\scriptstyle f_{ab}\otimes f_{ab}$};
   	\draw[->] (1) -- (4);
   	\draw[->] (2) -- (5);
   	\draw[->] (3) -- (6);
  \end{tikzpicture}
\]
is commutative because $f=f_{ab}\circ\pi$ is compatible with comultiplications, and we have just observed that the left hand square commutes. It follows that the precomposition of the right hand square with $\pi$ commutes also, and since $\pi$ is onto, the right hand square must be commutative. We again skip the entirely similar arguments for compatibility of $f_{ab}$ with counits and antipodes in the algebraic case. 
\end{proof}

Outside of the general categorical framework provided by \Cref{th.cqgfinpres,th.cstarqgsaft}, one can also arrive at limits in the categories $\cqg$ and $\cstarqg$ by means of the Tannakian formalism introduced in \cite{MR943923} ($\cqg$ is better suited for this, so we focus on it). There, Woronowicz associates a compact quantum group (or rather a compact quantum group of the form $\univ(A)$ for $A\in\cqg$, so effectively, he recovers a CQG algebra) from a so-called \define{concrete monoidal $W^*$-category with complex conjugation}. These are basically just rigid, monoidal $W^*$-categories endowed with a faithful, monoidal $*$-functor \cite{MR808930} to the category $\cat{Hilb}$ of finite-dimensional Hilbert spaces. 

This is a version of Tannaka duality for Hopf algebras (e.g. as in \cite{MR1623637,2012arXiv1202.3613V} and the many references therein): The CQG algebra constructed in \cite[1.3]{MR943923} given a concrete monoidal $W^*$-category $\mathcal C$ is what in \cite{MR1623637} would be called the coendomorphism Hopf algebra of the functor $\mathcal C\to\cat{Hilb}$ that is implicitly part of Woronowicz's definition. 

Now, if one starts with the category of unitary comodules of a CQG algebra $A$ and performs the above construction, the resulting CQG algebra is again $A$. In other words, unitary comodules know all there is to know about a CQG algebra (hence the name `Tannaka \define{re}construction'). It follows from this that the construction of new CQG algebras out of old (such as, say, the limit of some functor $F:J\to\cqg$ out of the values of $F$) has a chance of being carried out categorically: Identify the category of unitary comodules, and you know the CQG algebra. 

To get some insight into what limits in $\cqg$ look like, we do what the previous paragraph suggests, for products (but also state the result for pullbacks, as it will be useful in \Cref{se.mono}): Given a family $A_i$, $i\in I$ of objects in $\cqg$, what does the category of finite-dimensional, unitary comodules of the product $\displaystyle A=\prod_i A_i$ in $\cqg$ look like in terms of the categories of comodules of the individual $A_i$? The answer turns out to be quite simple; arguably simpler, in fact, than the description of the category of unitary comodules for the (more familiar, in the literature) coproduct $\displaystyle\coprod A_i$ from \cite{MR1316765}. All comodules below are understood to be finite-dimensional and unitary. 

Putting an $A$-comodule structure on an $n$-dimensional Hilbert space $V$ is the same as giving a CQG algebra morphism $f:A_u(Q)\to A$ for some positive $Q\in GL_n(\mathbb C)$: In one direction, the comodule structure induces a morphism by sending the standard generators $u_{ij}\in A_u(Q)$ (for an appropriate $Q$) to the coefficients $v_{ij}$ with respect to an orthonormal basis of $V$; in the opposite direction, make $A$ coact on the standard $A_u(Q)$-comodule by ``scalar corestriction'' via the coalgebra morphism $f$. In turn, by the defining property of the categorical product, a morphism $f:A_u(Q)\to A$ means a family of morphisms $f_i:A_u(Q)\to A_i$, $i\in I$. Going through this comodule structure - morphism correspondence in reverse for each $i$, the data consisting of the $f_i$'s is equivalent to putting an $A_i$-comodule structure on the canonical $n$-dimensional comodule of $A_u(Q)$ for every $i$. A moment's thought will show how do modify this argument to take care pullbacks, and all in all, we have the following result:

\begin{proposition}\label{prop.limTann}
Let $A_i\in\cqg$, $i\in I$ be a set of objects, and $A=\prod A_i$ the product in $\cqg$. Then, the category of $A$-comodules has as objects finite-dimensional Hilbert spaces admitting an $A_i$-comodule structure for each $A_i$, and as morphisms linear maps respecting all of these structures.

Let $f:B\to C$ and $f':B'\to C$ be morphisms in $\cqg$, and $\displaystyle A=B\times_CB'$ the pullback in $\cqg$. The category of $A$-comodules has 
\begin{enumerate}
\renewcommand{\labelenumi}{(\alph{enumi})}
\item as objects, triples $(V,V',\varphi)$ where $V$ and $V'$ are $B$ and $B'$-comodules respectively, and $\varphi:V\to V'$ is a unitary map identifying $V$ and $V'$ as $C$-comodules;
\item as morphisms from $(V,V',\varphi)$ to $(W,W',\psi)$, pairs $(\xi,\xi')$, where $\xi:V\to W$ and $\xi':V'\to W'$ are morphisms in $\mathcal M^B$ and $\mathcal M^{B'}$ respectively, making the diagram
\[
 \begin{tikzpicture}[anchor=base]
   \path (0,0) node (1) {$V$} +(2,0) node (2) {$V'$} +(0,-2) node (3) {$W$} +(2,-2) node (4) {$W'$}; 
   \draw[->] (1) -- (2) node[pos=.5,auto] {$\scriptstyle \varphi$};
   \draw[->] (3) -- (4) node[pos=.5,auto] {$\scriptstyle \psi$};
   \draw[->] (1) -- (3) node[pos=.5,auto,swap] {$\scriptstyle \xi$};
   \draw[->] (2) -- (4) node[pos=.5,auto] {$\scriptstyle \xi'$};
 \end{tikzpicture}
\]
commutative. 
\end{enumerate}\qedhere
\end{proposition}

\begin{remark}
This statement describes the sought-after categories of unitary comodules very explicitly. There is a more abstract, but also more elegant way to phrase all of this. We need some basic 2-categorical notions to say it all (as in \cite{MR2664622}).

Rigid, monoidal $W^*$-categories $\mathcal C$ endowed with monoidal $*$-functors $\mathcal C\to\cat{Hilb}$ form a bicategory in a natural way, while $\cqg$ can be regarded as a bicategory with only identity 2-cells. Then, sending a CQG algebra to its category of finite-dimensional, unitary comodules (together with its forgetful functor to $\cat{Hilb}$) is a pseudofunctor from the latter to the former. The essence of \Cref{prop.limTann} is that this pseudofunctor preserves limits. This is a familiar theme in Tannaka duality: Woronowicz's construction of a CQG algebra out of a functor $\mathcal C\to\cat{Hilb}$ is in fact nothing but the left adjoint of the pseudofunctor mentioned above. This sort of situation is treated in \cite{2011arXiv1112.5213S}, with a biadjunction analogous to the one just discussed appearing in Theorem 3.1.1.   
\end{remark}

\begin{remark}
The references to \cite{MR943923} in the above discussion are somewhat of a paraphrase, as Woronowicz works with what are called compact \define{matrix} quantum groups (or CMQG algebras on the algebraic side \cite[2.5]{MR1310296}). These are basically compact quantum groups finitely generated as $C^*$-algebras, and are analogous to compact Lie groups (the latter being precisely those compact groups which embed in some unitary group). He also distinguishes a comodule whose coefficients generate the algebra, and so works with pairs $(A,u)$, where $u\in M_n(A)$ is a unitary coefficient matrix. Adapting the results of that paper to the general setting is straightforward.

A CMQG algebra always has a countable set of (isomorphism classes of) simple comodules. As we will see in the next example, abandoning this restriction is absolutely necessary if we are going to discuss limits in $\cqg$, since products, for example, are very unlikely to satisfy this property.   
\end{remark}

\begin{example}
One does not even have to go ``quantum'' to give an example of a very large (that is, non-CMQG) product in $\cqg$. Indeed, the smallest possible example will do: a coproduct of two copies of $\mathbb Z/2$ in the category of compact groups. 

Denote this coproduct by $G$. According to the first part of \Cref{prop.limTann}, a unitary representation of $G$ consists of a finite-dimensional Hilbert space $V$ and two involutive unitary operators $x$ and $y$ on $V$. The projections $p=\frac{1+x}2$ and $q=\frac{1+y}2$ provide precisely the same information, so we work with them instead. $V$ is irreducible preciely when $p$ and $q$ have no common proper, non-zero invariant subspace. By the discussion carried out prior to the statement of \cite[Theorem 1.41]{MR1873025}, this is equivalent to the four pairwise infima $p\wedge q$, $p\wedge(1-q)$ etc. all being zero (where `$\wedge$' means orthogonal projection on the intersection of the ranges of the two projections). 

If $\dim V\ge 2$, the vanishing of all wedges $p\wedge q$, etc. makes it necessary that $V$ be even-dimensional and that $p$ and $q$ both have rank $\frac{\dim V}2$, but this is it: The set of pairs of projections satisfying the requirements is open dense in the set of all pairs of projections of rank $\frac{\dim V}2$ on $V$. Hence, if $\dim V=2n$ for some $n\ge 1$, the set of pairs $(p,q)$ that will make $V$ into an irreducible unitary $G$-representation is a manifold of dimension $4n^2$ (twice the dimension of the Grassmannian variety of $n$-dimensional subspaces of $V$ as a \define{real} manifold). We now have to quotienting by the equivalence relation $(p,q)\sim (upu^*,uqu^*)$ for unitary $u$, which accounts for isomorphic $G$-representations induced by different pairs of projections. Since the unitary group of $V$ has dimension $n(2n-1)$, this still leaves continuum many isomorphism classes of simples. 
\end{example}

We end this subsection with a note on terminology. As observed in \cite[$\S$3]{2010arXiv1010.3379S}, the name `free product of compact quantum groups' from \cite{MR1316765} is somewhat inconsistent with the prevailing point of view that compact quantum groups form a category opposite to $\cstarqg$. The problem is that in universal algebra, `free product' is often synonymous to `coproduct'. Even though Wang's construction is a coproduct in $\cstarqg$ rather than its opposite, and hence `\define{product} of compact quantum groups' would perhaps be a better fit, `free product' seems to have been established through use (besides, `product' would clash with the interpretation of $A\otimes A$ as the Cartesian square of a compact quantum group, implicit in \Cref{def.cstarqg}). On the other hand, products in $\cstarqg$ (whose existence \Cref{th.lim} proves) are probably best referred to as `coproducts of compact quantum groups'.


\subsection{Quantum groups generated by quantum spaces}\label{subse.spaces}

The idea here is that functors of the form ``forget the comultiplication'', regarded as quantum analogues of forgetting the multiplication on a group, have right adjoints. As most of the previous discussion, all of this works both algebraically and $C^*$-algebraically:

\begin{theorem}\label{th.spaces}
The functors
\begin{enumerate}
\renewcommand{\labelenumi}{(\alph{enumi})}
\item $\forget:\cqg\to\astaroh$ sending a CQG algebra to its underlying $*$-algebra and
\item $\forget:\cstarqg\to\cstaroh$ sending a compact quantum group to its underlying $C^*$-algebra
\end{enumerate}
are left adjoints.
\end{theorem}
\begin{proof}
We already know from the proofs of \Cref{prop.cqgcoco,prop.cstarqgcoco} that the forgetful functors from $\cqg$ and $\cstarqg$ to \define{unital} $*$-algebras and $C^*$-algebras respectively are cocontinuous. I claim that so are the inclusions $\astar\to\astaroh$ and $\cstar\to\cstaroh$. Assuming this for now, (a) and (b) are cocontinuous; that they are left adjoints then follows from the SAFT-ness of $\cqg$ and $\cstarqg$ (\Cref{th.cqgfinpres,th.cstarqgsaft}) and the adjoint functor theorem.

Finally, to prove the claim made above that the two inclusions $\astar\to\astaroh$ and $\cstar\to\cstaroh$ are cocontinuous, note that they are in fact left adjoints: In both cases, the right adjoint is the multiplier algebra construction $A\mapsto M(A)$. 
\end{proof}

\begin{remark}\label{rem.quantumcoll}
Von Neumann algebras would be an alternative way to formalize the idea of ``quantum space''. This is the point of view espoused in \cite{2012arXiv1202.2994K}, where the category $\wstar$ of unital von Neumann algebras and unital normal homomorphisms is opposite to the category of so-called quantum collections. The idea here is that a von Neumann algebra is a quantum analogue of a set, ordinary sets $X$ corresponding to $\ell^\infty(X)$. 

Adopting this perspective, the enveloping $W^*$-algebra functor $\env:\cstar\to\wstar$ is a kind of forgetful functor, disregarding the topological side of a compact quantum space and remembering only the underlying quantum collection; similarly, the forgetful functor $\wstar\to\cstar$ (which is right adjoint to $\env$) is a kind of quantum Stone-Cech compactification. 

Composing (b) of \Cref{th.spaces} further with $\env$ is again a left adjoint (the composition of two left adjoints), and its right adjoint could be thought of as the functor associating to every quantum ollection the compact quantum group freely generated by it. 
\end{remark}

We refer to the image of a $*$ or $C^*$-algebra $A$ through the respective right adjoint as the \define{cofree} CQG algebra or $\cstarqg$ object on $A$ (`co' because it universally maps into $A$, as opposed to being mapped into). The notation $\cof$ stands for either of the two functors, and we rely on context to distinguish between the possibilities.

It is to be expected in such cofree-Hopf-algebra-on-an-algebra type constructions that commutativity will be preserved (as explained, for instance, in the introduction of \cite{MR2745565}). In other words, one would like the right adjoint from part (b), say, when applied to the algebra of functions vanishing at infinity on a locally compact space $X$, to yield precisely the compact group freely generated by $X$: Construct the abstract group $G$ freely generated by $X$, endow it with the strongest topology making the canonical map $X\to G$ continuous, and take the Bohr compactification of the resulting topological group. That this is indeed the case is essentially the content of the following proposition:

\begin{proposition}\label{prop.spacescomm}
If $A$ is a commutative $*$ or $C^*$-algebra, then $\cof(A)$ is commutative. 
\end{proposition}
\begin{proof}
To fix ideas, we prove the algebraic statement regarding $*$-algebras, and leave the simple modifications that will adapt the proof to the other cases to the reader.     

Let $H=\cof(A)$, and recall the CQG algebra structure on $H_{ab}$ from \Cref{prop.limcomm}. The universality property of the abelianization factorizes the left hand diagonal arrow in the following diagram through the right hand one, while the cofree-ness gives the other commutative triangle, passing through $\iota$: 
\[
 \begin{tikzpicture}[anchor=base]
   \path (0,0) node (1) {$H$} +(4,0) node (2) {$H_{ab}$} +(2,-2) node (3) {$A$}; 
   \draw[->] (1) .. controls (1,.5) and (3,.5) .. (2) node[pos=.5,auto] {$\pi$};
   \draw[->] (2) .. controls (3,-.3) and (1,-.3) .. (1) node[pos=.5,auto] {$\iota$};
   \draw[->] (1) -- (3);
   \draw[->] (2) -- (3);
 \end{tikzpicture}
\]
By cofree-ness again, the loop $\iota\circ\pi$ must be the identity; since $\pi$ was a surjection, it must be an isomorphism.
\end{proof}


\subsection{Variations on the Bohr compactification theme}\label{subse.bohr}

A right adjoint to the inclusion functor $\cstarqg\to\cstarohqs$ is constructed directly in \cite{MR2210362}, and this construction is referred to as the \define{quantum Bohr compactification}. For any $A\in\cstarohqs$, it provides an object $H\in\cstarqg$ mapping universally into $A$ so as to preserve the comultiplication; remembering the arrow reversal inherent to passing from spaces to functions on them, this should indeed be thought of as a compact quantum group into which the locally compact quantum semigroup maps universally. Moreover, when $A$ is commutative and hence the algebra of functions vanishing at infinity on a locally compact semigroup $X$, the construction returns precisely the algebra of functions on the ordinary Bohr compactification of $X$ (\cite[4.1]{MR2210362}). 

We recover these results and their algebraic counterparts below (\Cref{th.bohr} and \Cref{prop.bohrcomm}), as applications of the categorical machinery already in place.

\begin{theorem}\label{th.bohr}
The inclusion functors
\begin{enumerate}
\renewcommand{\labelenumi}{(\alph{enumi})}
\item $\forget:\cqg\to\astarohqs$ sending a CQG algebra to its underlying algebraic quantum semigroup and
\item $\forget:\cstarqg\to\cstarohqs$ sending a compact quantum group to its underlying compact quantum group
\end{enumerate}
are left adjoints.
\end{theorem}
\begin{proof}
As before, \Cref{th.cqgfinpres,th.cstarqgsaft} and the adjoint functor theorem ensure that we only need to prove the two functors cocontinuous. Equivalently, this means showing they preserve coequalizers of pairs and coproducts. The four arguments (coequalizers and coproducts, (a) and (b)) follow essentially the same path, so let us focus on coproducts for part (a).

Let $I$ be a set, and $A_i$, $i\in I$ objects in $\cqg$. Let also $f_i:A_i\to B$ be morphisms in $\astarohqs$ (strictly speaking, they are morphisms $\forget(A_i)\to B$, but since $\forget$ really is just an inclusion, we omit it in the rest of the proof). Since forgetting further to $\astaroh$ (i.e. disregarding comultiplications) is, according to part (a) of \Cref{th.spaces}, cocontinuous, the $f_i$ aggregate into a unique $*$-algebra morphism $f:A=\coprod A_i\to B$. We are done if we can show that $f$ preserves comultiplications. To see this, consider the diagram
\[
	\begin{tikzpicture}[anchor=base]
   	\path (0,0) node (1) {$A_i$} +(3,0) node (2) {$A$} +(6,0) node (3) {$B$} +(0,-2) node (4) {$A_i\otimes A_i$} +(3,-2) node (5) {$A\otimes A$} +(6,-2) node (6) {$B\otimes B$}; 
   	\draw[->] (1) -- (2) node[pos=.5,auto] {$\scriptstyle \iota_i$};
   	\draw[->] (2) -- (3) node[pos=.5,auto] {$\scriptstyle f$};
   	\draw[->] (4) -- (5) node[pos=.5,auto] {$\scriptstyle \iota_i\otimes \iota_i$};
   	\draw[->] (5) -- (6) node[pos=.5,auto] {$\scriptstyle f\otimes f$};
   	\draw[->] (1) -- (4);
   	\draw[->] (2) -- (5);
   	\draw[->] (3) -- (6);
  \end{tikzpicture}
\]
where the vertical arrows are comultiplications, and $\iota_i:A_i\to A$ are the structure maps of the coproduct. The commutativity of the outer rectangle is the preservation of comultiplications by the $f_i=f\circ\iota_i$, while the left hand square commutes because $A$ was defined as the coproduct of the $A_i$ in $\cqg$. It follows that precomposing the two possible ways to get from $A$ to $B\otimes B$ with $\iota_i$ yields the same morphism $A_i\to B\otimes B$; then, by the universality in $\astaroh$ of the coproduct $A$, the right hand square must also be commutative.    
\end{proof}

\begin{remark}\label{rem.quantumcollbis}
Part (b) of the proposition can be tweaked slightly in the spirit of \Cref{rem.quantumcoll}. 

The enveloping von Neumann algebra functor $\env:\cstaroh\to\wstar$ can be lifted to a functor (again called $\env$) from $\cstarohqs$ to the category $\wstarqs$ of \define{von Neumann quantum semigroups}, consisting of von neumann algebras $M$ endowed with a coassociative morphism $\Delta:M\to M\otimes M$ in $\wstar$ (remember that the tensor product of von Neumann algebras here is the spatial one) and with $\wstar$ maps which preserve these comultiplications as morphisms. It is to be thought of as a kind of forgetful functor, ignoring the topology of a locally compact quantum semigroup and remembering only the underlying quantum collection, together with the ``multiplication''.

It can be shown further that $\env\circ\forget:\cstarqg\to\wstarqs$ is cocontinuous, and hence a left adjoint. In other words, every $W^*$ quantum semigroup has a quantum Bohr compactification.  
\end{remark}

\begin{remark}\label{rem.noleftadj}
In a $C^*$-algebraic variant of the Tambara construction \cite{MR1071429}, (a particular case of) \cite[Theorem 3.3]{MR2501746} constructs, for every finite-dimensional $C^*$-algebra $A$, an object $B$ of $\cstarqs$ (the full subcategory of $\cstarohqs$ consisting of unital algebras) coacting universally on $A$. In other words, there is a coassociative $C^*$-algebra map $A\to A\otimes B$ making $B$ an initial object in the category of objects of $\cstarqs$ endowed with such maps. 

If $B$ were to map universally into an object $B'\in\cstarqg$, the latter would be the \define{quantum automorphism group} of $A$ in the sense of \cite{MR1637425}. However, we know from \cite[Theorem 6.1 (1)]{MR1637425} that finite-dimensional $C^*$-algebras do not have compact quantum automorphism groups, in general. In conclusion, although $\cstarqg\to\cstarqs$ is a left adjoint (by a variant of \Cref{th.bohr}), it is \define{not} a right adjoint. 
\end{remark}

As in \Cref{subse.spaces}, we denote the right adjoints of \Cref{th.bohr} by $\cof$. Once more, it turns out that they preserve commutativity.

\begin{proposition}\label{prop.bohrcomm}
If the object $A$ of $\astarohqs$ or $\cstarohqs$ is commutative, so is $\cof(A)$.  
\end{proposition}
\begin{proof}
We focus on the $\astarohqs$ case.

Setting $H=\cof(A)$, the canonical map $f:H\to A$ factors as
\[
 \begin{tikzpicture}[anchor=base]
   \path (0,0) node (1) {$H$} +(4,0) node (2) {$H_{ab}$} +(2,-2) node (3) {$A$}; 
   \draw[->] (1) -- (2) node[pos=.5,auto] {$\scriptstyle \pi$};
   \draw[->] (1) -- (3) node[pos=.5,auto,swap] {$\scriptstyle f$};
   \draw[->] (2) -- (3) node[pos=.5,auto] {$\scriptstyle f_{ab}$};
 \end{tikzpicture}
\]
for some morphism $f_{ab}$ in $\astaroh$. Essentially the same argument as the one used in the proof of \Cref{prop.limcomm} shows that $f_{ab}$ is actually a morphism in $\astarohqs$. Finally, we can now repeat the proof of \Cref{prop.spacescomm} to conclude that $\pi$ is in fact an isomorphism, and hence $H$ is commutative. 
\end{proof}

\begin{remark}
\Cref{prop.bohrcomm} goes through in the setting of \Cref{rem.quantumcollbis}: If $M$ is a commutative von Neumann algebra, then it can be shown in much the same way as above that the quantum Bohr compactification is commutative. 
\end{remark}


\subsection{Kac quotients}\label{subse.kac}

Recall that a compact quantum group $A\in\cstarqg$ is said to be \define{of Kac type} if the antipode on $\alg(A)$ lifts to a continuous map $A\to A$. Equivalently, the antipode of $\alg(A)$ is involutive ($S^2=\id$), or commutes with the $*$ operation. This definition extends in the obvious way to CQG algebras; in that case, `of Kac type' or simply `Kac' will be synonymous to `having involutive antipode'. Alternate terms are `Kac algebra' (under which these objects and their relatives were introduced; e.g. \cite{MR1215933} and the references therein) or sometimes `Woronowicz-Kac algebra' (as in \cite{MR1734250}). For brevity, we will sometimes simply use `Kac' as an adjective. A $k$ subscript on either $\cqg$ of $\cstarqg$ indicates the full subcategory on objects of Kac type.

In \cite[Appendix]{MR2210362}, So\l tan constructs what in that paper is called \define{the canonical Kac quotient} of an object $A\in\cstarqg$ (notion originally due to of Stefaan Vaes). It is obtained by quotienting out all elements of $A$ killed by some trace (meaning, as usual, that the trace sends $x^*x$ to zero). This, however, seems to be a bit of a misnomer: While it is shown in \cite[A.1]{MR2210362} that the result is indeed a compact quantum group of Kac type, it is not clear that a Kac compact quantum group is its own canonical Kac quotient\footnote{I am grateful to Piotr So\l tan for pointing this out.}! A more appropriate term might be, perhaps, canonical \define{tracial} quotient: One quotients out as much as one needs to in order to ensure that the result has enough traces. 

The notion has also received attention in \cite{MR2335776}, where Tomatsu shows in Theorem 4.8 that a compact quantum group $A$ has a largest quotient of Kac type (he uses dual phrasing, thinking of the latter as a largest compact quantum \define{subgroup} of Kac type) provided $A$ is coamenable and the Grothendieck ring of its category of comodules (its so-called fusion algebra) is commutative. Regarding the terminology problem from the previous paragraph, note that for the reason pointed out there, it is not clear, a priori, that Tomatsu's quotient is the same as So\l tan's. The two do coincide, however, if the Haar measure of the compact quantum group is faithful (which is the standing assumption of \cite{MR2335776}), hence \cite[Remark 4.9]{MR2335776}. 

In conclusion, the question of whether or not \define{every} $A\in\cqg$ or $\cstarqg$ has a largest Kac quotient seems to be an interesting one. One of the aims of this subsection is to show that this is indeed the case: The desired quotient map is precisely the reflection of $A$ in the subcategory $\cqg_k$ or $\cstarqg_k$ (i.e. the image of $A$ through the left adjoint to the inclusion of the subcategory into $\cqg$ or $\cstarqg$, respectively).

\begin{theorem}\label{th.kac}
The inclusions $\cqg_k\to\cqg$ and $\cstarqg_k\to\cstarqg$ each have a left, as well as a right adjoint. 
\end{theorem}
\begin{proof}
That the inclusions are both left adjoints is shown in much the same way in which we have proven all results asserting the existence of various right adjoints so far. The arguments of \Cref{se.cqgfinpres,se.cstarqgsaft} can be repeated to show that $\cqg_k$ is finitely presentable and $\cstarqg_k$ is SAFT. The only difference that is even remotely significant is the fact that now the regular generator to go into an analogue of \Cref{prop.reggen} consists of $A_u(I_n)$, $n\ge 1$ rather than all $A_u(Q)$. Colimits are again constructed simply at the level of $*$ or $C^*$-algebras, making it clear that the inclusions are cocontinuous and hence left adjoints.

The interesting problem, then, is the one discussed before the statement of the theorem: constructing \define{left} adjoints to the two inclusions. For each $A\in\cqg$ or $\cstarqg$ we want an arrow $\kappa:A\to A_k$ into a Kac type object, universal in the sense that any morphism from $A$ into a Kac object factors uniquely through $\kappa$. In other words, we have to show that the comma category $A\downarrow\cqg$ (resp. $A\downarrow\cstarqg$) consisting of arrows from $A$ into Kac objects has an initial object. To do this, we apply Freyd's initial object theorem \cite[V.6.1]{MR1712872}. It says that a complete category has an initial object as soon as it has a \define{weakly initial set} of objects; this simply means a set $S$ of objects such that any object $y$ admits at least one arrow $S\ni x\to y$ (not necessarily unique). 

In our case, a weakly initial set is easy to come by: All surjections $A\to B$ for Kac type $B$ will do, since any map of $A$ into a Kac type object will certainly factor through the image of that map. On the other hand completeness follows quickly if we show that $\cqg_k$ and $\cstarqg_k$ are closed under limits in $\cqg$ and $\cstarqg$ respectively: Limits would then be created by the forgetful functor $A\downarrow\cqg_k\to\cqg_k\to\cqg$ (and similarly for $\cstarqg$).  

In conclusion, it is enough to show that products of Kac objects in $\cqg$ (or $\cstarqg$) are again Kac, and similarly, equalizers of parallel pairs of arrows between Kac objects are Kac. Since (a) $\alg:\cstarqg\to\cqg$ is a right adjoint and hence preserves limits, and (b) by definition, an object $B\in\cstarqg$ is Kac if and only if $\alg(B)$ is, it is enough to restrict ourselves to $\cqg$. 

Equalizers are easy: If $f,g:B\to C$ are arrows between Kac CQG algebras, the equalizer injects into $B$, so it is again Kac. To prove the statement about products, let $B_i$, $i\in I$ be a set of Kac CQG algebras, and denote the structure maps of their product in $\cqg$ by $\pi_i:B\to B_i$. Throughout the rest of this proof, for a CQG algebra $C$, we denote by $C'$ the CQG algebra with the same underlying set, but reversed multiplication and comultiplication. Note that the product of the objects $B_i'$ is precisely $B'$. Let $S_i$ be the antipodes of $B_i$, and $S=\prod S_i:B\to B'$ the map obtained from the functoriality of products. By this same functoriality, $S$ is involutive (strictly speaking, this means $S'\circ S=\id$, where $S':B'\to B$ is $S$ as a map, but we have switched the domain and codomain). If we show that $S$ is the antipode $S_B$ of $B$, we are done (Kac means involutive). To prove this, note that the two squares in the diagram 
\[
 \begin{tikzpicture}[anchor=base]
   \path (0,0) node (1) {$B$} +(2,0) node (2) {$B'$} +(0,-2) node (3) {$B_i$} +(2,-2) node (4) {$B_i'$}; 
   \draw[->] (1) .. controls (1,.3) and (1,.3) .. (2) node[pos=.5,auto] {$\scriptstyle S$};
   \draw[->] (1) .. controls (1,-.1) and (1,-.1) .. (2) node[pos=.5,auto,swap] {$\scriptstyle S_B$};
   \draw[->] (3) -- (4) node[pos=.5,auto] {$\scriptstyle S_i$};
   \draw[->] (1) -- (3) node[pos=.5,auto,swap] {$\scriptstyle \pi_i$};
   \draw[->] (2) -- (4) node[pos=.5,auto] {$\scriptstyle \pi_i$};
 \end{tikzpicture}
\]
are both commutative (the $S$-square by the definition of $S$ as the product of the $S_i$, and the $S_B$-square because $\pi_i$ are Hopf algebra morphisms and hence preserve antipodes).
\end{proof}

Denote by $A\mapsto A_k$ the left adjoint to either of the inclusions $\cqg_k\to\cqg$ or $\cstarqg_k\to\cstarqg$. It is a simple observation now that the canonical map $A\to A_k$ is a surjection, and hence warrants the name `Kac \define{quotient}':

\begin{proposition}
Let $A\in\cqg$ or $\cstarqg$ and $\kappa:A\to A_k$ the universal arrow resulting from the unit of the adjunction that the Kac quotient functor is part of. Then, $\kappa$ is onto. 
\end{proposition}
\begin{proof}
To keep things streamlined, let us focus on $\cqg$. We have seen this sort of argument before, in a dual form, in \Cref{prop.spacescomm}. Let $\iota:A_k'\to A_k$ be the inclusion of $A_k'=\mathrm{Im}(\kappa)$, and denote the corestriction by $\kappa':A\to A_k'$. We have the diagram
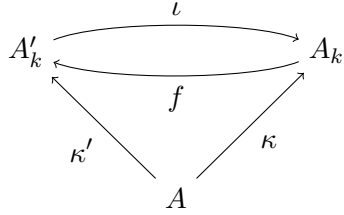
\[
 \begin{tikzpicture}[anchor=base]
   \path (0,0) node (1) {$A_k'$} +(4,0) node (2) {$A_k$} +(2,-2) node (3) {$A$}; 
   \draw[->] (1) .. controls (1,.5) and (3,.5) .. (2) node[pos=.5,auto] {$\displaystyle \iota$};
   \draw[->] (2) .. controls (3,-.3) and (1,-.3) .. (1) node[pos=.5,auto] {$\displaystyle f$};
   \draw[<-] (1) -- (3) node[pos=.5,auto,swap] {$\displaystyle \kappa'$};
   \draw[<-] (2) -- (3) node[pos=.5,auto] {$\displaystyle \kappa$};
 \end{tikzpicture}
\]
where $f$ is the unique arrow $\kappa\to\kappa'$ in $A\downarrow\cqg_k$. Both triangles are commutative, and by the universality of $\kappa$, the loop $\iota\circ f$ must be the identity. Since $\iota$ was by definition one-to-one, it must be an isomorphism.
\end{proof}


\section{Monomorphisms}\label{se.mono}

One subject is conspicuously absent from \Cref{prop.epi}: what about monomorphisms? The main result of this section is just such a characterization:

\begin{proposition}\label{prop.mono}
A morphism in $\cqg$ is a monomorphism if and only if it is one-to-one. Similarly, a morphism $f$ in $\cstarqg$ is mono if and only if $\alg(f)$ is one-to-one.  
\end{proposition}

\begin{remark}
Morphisms in $\cstarqg$ that are injective at the algebraic level play an important role in \cite[6.2]{MR2805748}. The proposition gives a nice interpretation: They are precisely the monomorphisms of $\cstarqg$. 
\end{remark}

\begin{remark}
In the commutative setting, where all distinctions between the algebraic and the $C^*$-algebraic sides of the picture disappear (\Cref{rem.cqgab}), the analogous result would be that epimorphisms of compact groups are surjective. This is \cite[Proposition 9]{MR0260829}, and in fact, the proof below is a paraphrase of Reid's. 
\end{remark}

\begin{proof}
First, let's reduce the second part of the statement to the first. Recall that right adjoints (such as $\alg$) send monomorphisms to monomorphisms, so a morphism $f$ in $\cstarqg$ can only be mono if $\alg(f)$ is. On the other hand, the faithfulness of $\alg$ (a consequence of the density of $\alg(A)\subset A$ for $A\in\cstarqg$) implies the converse. Indeed, if $\alg(f)$ is mono and $f\circ g=f\circ h$, the series \[ \alg(f)\circ\alg(g)=\alg(f)\circ\alg(h)\quad \Rightarrow\quad \alg(g)=\alg(h)\quad \Rightarrow\quad g=h \] of equalities does the trick (the first implication says that $\alg(f)$ is a monomorphism, while the second one is faithfulness).

We are now left with the first statement, on CQG algebras. Just as in the case of epimorphisms treated in \Cref{prop.epi}, the implication injective $\Rightarrow$ mono is the easy part. Focusing on the converse, let $f:A\to B$ be a monomorphism in $\cqg$. Let also $\pi:P=A\times_BA\to A$ be one of the defining projections of the pullback (in the category $\cqg$). The functor $\cqg(\bullet,P)$ represents the functor sending $C\in\cqg$ to the set of pairs of morphisms $g,h:C\to A$ satisfying the condition $fg=fh$. Since the latter implies $g=h$ by $f$ being a monomorphism, this functor is isomorphic to $\cqg(\bullet,A)$, and it follows that the natural transformation $\cqg(\bullet,P)\to\cqg(\bullet, A)$ induced by $ \pi$ (and hence $ \pi$ itself) is an isomorphism; we have made use of this sort of result, in its dual form having to do with epimorphisms, in \Cref{prop.epi}. It will actually be more convenient to say it like this: The diagonal map $d:A\to A\times_BA=P$ is an isomorphism; indeed, its very definition implies that $\pi d=\id$.  

Now transport all of the above to the level of (finite-dimensional, unitary) comodules. Recalling what the category of $P$-comodules looks like from part (b) of \Cref{prop.limTann}, the fact $d$ is an isomorphism says that $V\mapsto (V,V,\id)$ is an equivalence from $A$-comodules to $P$-comodules. In particular, the essential surjectivity of this functor implies that for any $P$-comodule $(V,V',\varphi)$, the a priori $B$-comodule isomorphism $\varphi:V\to V'$ is actually an $A$-comodule map. 

According to the previous paragraph, it is enough, assuming that $f$ is not one-to-one, to find finite-dimensional, unitary $A$-comodules $V$ and $V'$ together with a unitary isomorphism $\varphi:V\to V'$ as $B$-comodules which does \define{not} preserve the $A$-comodule structures. Let us simplify the situation further. Suppose $V$ is a non-trivial, simple, unitary $A$-comodule (non-trivial meaning not isomorphic to the monoidal unit of the cateory of comodules) which has trivial components when regarded as a $B$-comodule via scalar corestriction through $f:A\to B$. This means that there is some non-zero vector $v\in V$ fixed by $B$ in the sense that $v_1\otimes f(v_2)=v\otimes 1$, but not fixed by $A$. The unitary reflection across the orthogonal complement of $v$ would then be a morphism in $\mathcal M^B$ but not in $\mathcal M^A$, and we would be done.   

In conclusion, it suffices to find $V$ as above. Since the Peter-Weyl theorem for CQG algebras expresses each as a direct sum of $W^{\oplus \dim W}$ for $W$ ranging over its set of unitary simple comodules, the only ways in which $f$ can be non-injective are if (a) some simple $A$-comodule becomes non-simple as a $B$-comodule, or (b) there are two distinct simples over $A$ which become isomorphic over $B$ (we will see soon that in fact (a) always happens). 

In case (a), choose some such simple $W\in\widehat{A}$. Then, the trivial comodule has multiplicity one in $W^*\otimes W$ as an $A$-comodule, but strictly larger than one over $B$. Hence, there is a simple $V\le W^*\otimes W$ that will satisfy the sought-for conditions. In case (b), let $W$ and $W'$ be non-isomorphic, unitary simple $A$-comodules which bcome isomorphic over $B$. Then, $W^*\otimes W'$ does not contain the trivial comodule over $A$, but it does over $B$. So as before, we can find our desired $V$ among the simple summands of $W^*\otimes W'$.
\end{proof}



\addcontentsline{toc}{section}{References}



\begin{thebibliography}{{Kor}12}

\bibitem[Abe80]{MR594432}
Eiichi Abe.
\newblock {\em Hopf algebras}, volume~74 of {\em Cambridge Tracts in
  Mathematics}.
\newblock Cambridge University Press, Cambridge, 1980.
\newblock Translated from the Japanese by Hisae Kinoshita and Hiroko Tanaka.

\bibitem[AR94]{MR1294136}
Ji{\v{r}}{\'{\i}} Ad{\'a}mek and Ji{\v{r}}{\'{\i}} Rosick{\'y}.
\newblock {\em Locally presentable and accessible categories}, volume 189 of
  {\em London Mathematical Society Lecture Note Series}.
\newblock Cambridge University Press, Cambridge, 1994.

\bibitem[Ban99]{MR1734250}
Teodor Banica.
\newblock Fusion rules for representations of compact quantum groups.
\newblock {\em Exposition. Math.}, 17(4):313--337, 1999.

\bibitem[Ban05]{MR2174219}
Teodor Banica.
\newblock Quantum automorphism groups of small metric spaces.
\newblock {\em Pacific J. Math.}, 219(1):27--51, 2005.

\bibitem[Bic03]{MR1937403}
Julien Bichon.
\newblock Quantum automorphism groups of finite graphs.
\newblock {\em Proc. Amer. Math. Soc.}, 131(3):665--673 (electronic), 2003.

\bibitem[BKQ11]{MR2805748}
Erik B{\'e}dos, S.~Kaliszewski, and John Quigg.
\newblock Reflective-coreflective equivalence.
\newblock {\em Theory Appl. Categ.}, 25:No. 6, 142--179, 2011.

\bibitem[BMT01]{MR1862084}
E.~B{\'e}dos, G.~J. Murphy, and L.~Tuset.
\newblock Co-amenability of compact quantum groups.
\newblock {\em J. Geom. Phys.}, 40(2):130--153, 2001.

\bibitem[Boc95]{MR1372527}
Florin~P. Boca.
\newblock Ergodic actions of compact matrix pseudogroups on {$C^*$}-algebras.
\newblock {\em Ast\'erisque}, (232):93--109, 1995.
\newblock Recent advances in operator algebras (Orl{\'e}ans, 1992).

\bibitem[{Chi}11]{2011arXiv1110.6701C}
A.~{Chirvasitu}.
\newblock {Cosemisimple Hopf algebras are faithfully flat over Hopf
  subalgebras}.
\newblock {\em ArXiv e-prints}, October 2011.

\bibitem[Con94]{Connes1994}
Alain Connes.
\newblock {\em Noncommutative Geometry}.
\newblock Academic Press, San Diego, CA, 1994.
\newblock available for download at \url{http://www.alainconnes.org/en/}.

\bibitem[DK94]{MR1310296}
Mathijs~S. Dijkhuizen and Tom~H. Koornwinder.
\newblock C{QG} algebras: a direct algebraic approach to compact quantum
  groups.
\newblock {\em Lett. Math. Phys.}, 32(4):315--330, 1994.

\bibitem[ES92]{MR1215933}
Michel Enock and Jean-Marie Schwartz.
\newblock {\em Kac algebras and duality of locally compact groups}.
\newblock Springer-Verlag, Berlin, 1992.
\newblock With a preface by Alain Connes, With a postface by Adrian Ocneanu.

\bibitem[GLR85]{MR808930}
P.~Ghez, R.~Lima, and J.~E. Roberts.
\newblock {$W^\ast$}-categories.
\newblock {\em Pacific J. Math.}, 120(1):79--109, 1985.

\bibitem[Kel86]{MR850527}
G.~M. Kelly.
\newblock A survey of totality for enriched and ordinary categories.
\newblock {\em Cahiers Topologie G\'eom. Diff\'erentielle Cat\'eg.},
  27(2):109--132, 1986.

\bibitem[{Kor}12]{2012arXiv1202.2994K}
A.~{Kornell}.
\newblock {Quantum Collections}.
\newblock {\em ArXiv e-prints}, February 2012.

\bibitem[KS97]{MR1492989}
Anatoli Klimyk and Konrad Schm{\"u}dgen.
\newblock {\em Quantum groups and their representations}.
\newblock Texts and Monographs in Physics. Springer-Verlag, Berlin, 1997.

\bibitem[KT99]{MR1741102}
Johan Kustermans and Lars Tuset.
\newblock A survey of {$C^*$}-algebraic quantum groups. {I}.
\newblock {\em Irish Math. Soc. Bull.}, (43):8--63, 1999.

\bibitem[KV00]{MR1832993}
Johan Kustermans and Stefaan Vaes.
\newblock Locally compact quantum groups.
\newblock {\em Ann. Sci. \'Ecole Norm. Sup. (4)}, 33(6):837--934, 2000.

\bibitem[Lac10]{MR2664622}
Stephen Lack.
\newblock A 2-categories companion.
\newblock In {\em Towards higher categories}, volume 152 of {\em IMA Vol. Math.
  Appl.}, pages 105--191. Springer, New York, 2010.

\bibitem[ML98]{MR1712872}
Saunders Mac~Lane.
\newblock {\em Categories for the working mathematician}, volume~5 of {\em
  Graduate Texts in Mathematics}.
\newblock Springer-Verlag, New York, second edition, 1998.

\bibitem[Mon93]{MR1243637}
Susan Montgomery.
\newblock {\em Hopf algebras and their actions on rings}, volume~82 of {\em
  CBMS Regional Conference Series in Mathematics}.
\newblock Published for the Conference Board of the Mathematical Sciences,
  Washington, DC, 1993.

\bibitem[Por11]{MR2745565}
Hans-E. Porst.
\newblock Limits and colimits of {H}opf algebras.
\newblock {\em J. Algebra}, 328:254--267, 2011.

\bibitem[PW90]{MR1059324}
P.~Podle{\'s} and S.~L. Woronowicz.
\newblock Quantum deformation of {L}orentz group.
\newblock {\em Comm. Math. Phys.}, 130(2):381--431, 1990.

\bibitem[QS10]{2010arXiv1007.0363Q}
J.~{Quaegebeur} and M.~{Sabbe}.
\newblock {Isometric coactions of compact quantum groups on compact quantum
  metric spaces}.
\newblock {\em ArXiv e-prints}, July 2010.

\bibitem[Rei70]{MR0260829}
G.~A. Reid.
\newblock Epimorphisms and surjectivity.
\newblock {\em Invent. Math.}, 9:295--307, 1969/1970.

\bibitem[Sch92]{MR1623637}
Peter Schauenburg.
\newblock Tannaka duality for arbitrary {H}opf algebras.
\newblock 66:ii+57, 1992.

\bibitem[{Sch}11]{2011arXiv1112.5213S}
D.~{Sch{\"a}ppi}.
\newblock {The formal theory of Tannaka duality}.
\newblock {\em ArXiv e-prints}, December 2011.

\bibitem[So{\l}05]{MR2210362}
Piotr~M. So{\l}tan.
\newblock Quantum {B}ohr compactification.
\newblock {\em Illinois J. Math.}, 49(4):1245--1270 (electronic), 2005.

\bibitem[So{\l}09]{MR2501746}
Piotr~M. So{\l}tan.
\newblock Quantum families of maps and quantum semigroups on finite quantum
  spaces.
\newblock {\em J. Geom. Phys.}, 59(3):354--368, 2009.

\bibitem[{Sol}10]{2010arXiv1010.3379S}
P.~M. {Soltan}.
\newblock {On quantum maps into quantum semigroups}.
\newblock {\em ArXiv e-prints}, October 2010.

\bibitem[Swe69]{MR0252485}
Moss~E. Sweedler.
\newblock {\em Hopf algebras}.
\newblock Mathematics Lecture Note Series. W. A. Benjamin, Inc., New York,
  1969.

\bibitem[Tak02]{MR1873025}
M.~Takesaki.
\newblock {\em Theory of operator algebras. {I}}, volume 124 of {\em
  Encyclopaedia of Mathematical Sciences}.
\newblock Springer-Verlag, Berlin, 2002.
\newblock Reprint of the first (1979) edition, Operator Algebras and
  Non-commutative Geometry, 5.

\bibitem[Tam90]{MR1071429}
D.~Tambara.
\newblock The coendomorphism bialgebra of an algebra.
\newblock {\em J. Fac. Sci. Univ. Tokyo Sect. IA Math.}, 37(2):425--456, 1990.

\bibitem[Tom07]{MR2335776}
Reiji Tomatsu.
\newblock A characterization of right coideals of quotient type and its
  application to classification of {P}oisson boundaries.
\newblock {\em Comm. Math. Phys.}, 275(1):271--296, 2007.

\bibitem[VD94]{MR1220906}
A.~Van~Daele.
\newblock Multiplier {H}opf algebras.
\newblock {\em Trans. Amer. Math. Soc.}, 342(2):917--932, 1994.

\bibitem[VDW96]{MR1382726}
Alfons Van~Daele and Shuzhou Wang.
\newblock Universal quantum groups.
\newblock {\em Internat. J. Math.}, 7(2):255--263, 1996.

\bibitem[{Ver}12]{2012arXiv1202.3613V}
J.~{Vercruysse}.
\newblock {Hopf algebras---Variant notions and reconstruction theorems}.
\newblock {\em ArXiv e-prints}, February 2012.

\bibitem[Wan95]{MR1316765}
Shuzhou Wang.
\newblock Free products of compact quantum groups.
\newblock {\em Comm. Math. Phys.}, 167(3):671--692, 1995.

\bibitem[Wan97]{MR1424670}
Shuzhou Wang.
\newblock Krein duality for compact quantum groups.
\newblock {\em J. Math. Phys.}, 38(1):524--534, 1997.

\bibitem[Wan98]{MR1637425}
Shuzhou Wang.
\newblock Quantum symmetry groups of finite spaces.
\newblock {\em Comm. Math. Phys.}, 195(1):195--211, 1998.

\bibitem[Wan99]{MR1697607}
Shuzhou Wang.
\newblock Ergodic actions of universal quantum groups on operator algebras.
\newblock {\em Comm. Math. Phys.}, 203(2):481--498, 1999.

\bibitem[Was94]{MR1271145}
Simon Wassermann.
\newblock {\em Exact {$C^*$}-algebras and related topics}, volume~19 of {\em
  Lecture Notes Series}.
\newblock Seoul National University Research Institute of Mathematics Global
  Analysis Research Center, Seoul, 1994.

\bibitem[Wor87]{Woronowicz1987}
S.~L. Woronowicz.
\newblock Compact matrix pseudogroups.
\newblock {\em Comm. Math. Phys.}, 111(4):613--665, 1987.

\bibitem[Wor88]{MR943923}
S.~L. Woronowicz.
\newblock Tannaka-{K}re\u\i n duality for compact matrix pseudogroups.
  {T}wisted {${\rm SU}(N)$} groups.
\newblock {\em Invent. Math.}, 93(1):35--76, 1988.

\bibitem[Wor98]{MR1616348}
S.~L. Woronowicz.
\newblock Compact quantum groups.
\newblock In {\em Sym\'etries quantiques ({L}es {H}ouches, 1995)}, pages
  845--884. North-Holland, Amsterdam, 1998.

\end{thebibliography}

\end{document}